\documentclass[1p]{elsarticle}

\usepackage{lineno}
\modulolinenumbers[5]
\journal{  }

\usepackage{mathrsfs}
\usepackage{amsfonts}
\usepackage{enumerate}
\usepackage{latexsym}
\usepackage[all,cmtip]{xy}
\usepackage{verbatim}
\usepackage{color}

\usepackage[colorlinks,linkcolor=blue,citecolor=blue]{hyperref}

\usepackage{amsmath,amssymb,xspace,amsthm}

\numberwithin{equation}{section}

\def\gl{\mathfrak{gl}}
\def\sl{\mathfrak{sl}}
\def\fa{\mathfrak{a}}

\def\dg{\dot{\mathfrak{g}}}
\def\ddg{\ddot{\mathfrak{g}}}
\def\fg{\mathfrak{g}}

\def\fL{\mathfrak{L}}
\def\ca{\mathfrak{a}}
\def\fK{\mathfrak{K}}

\def\Supp{\mathrm{Supp}}

\def\mod{\rm mod\,\,}

\def\bC{\mathbb{C}}
\def\bN{\mathbb{N}}
\def\bZ{\mathbb{Z}}
\def\bQ{\mathbb{Q}}
\def\Inv{{\rm Inv}}

\def\Der{\mathrm{Der}}

\def\supp{\mathrm{supp}}
\def\ann{\mathrm{ann}}

\def\Vir{\mathrm{Vir}}
\def\Hom{\mathrm{Hom}}

\def\Ind{{\rm Ind}}

\def\Bil{{\rm Bil}}

\def\tL{\widetilde{\fL}}

\def\fl{\mathfrak{l}}
\def\hL{\hat{\fL}}
\def\so{\mathfrak{so}}

\newtheorem{theo}{{Theorem}}
\newtheorem{lemm}[theo]{Lemma}
\newtheorem{remark}[theo]{Remark}

\newtheorem{prop}[theo]{Proposition}
\newtheorem{exam}{Example}

\allowdisplaybreaks

\begin{document}

\begin{frontmatter}



\title{$(d,\sigma)$-twisted Affine-Virasoro superalgebras}


\author{Rencai L\"u, Xizhou You and Kaiming Zhao}

\begin{abstract}
For any finite dimensional Lie superalgebra $\dg$ (maybe a Lie algebra) with an even derivation $d$ and a finite order automorphism $\sigma$ that commutes with $d$, we introduce  the $(d,\sigma)$-twisted Affine-Virasoro superalgebra $\fL=\fL(\dg,d,\sigma)$ and determine its universal central extension $\hL=\hL(\dg,d,\sigma)$. This is a huge class of infinite-dimensional Lie superalgebras. Such Lie superalgebras consist of  many new and  well-known Lie algebras and superalgebras, including the   Affine-Virasoro superalgebras,  the twisted Heisenberg-Virasoro algebra, the mirror Heisenberg-Virasoro algebra, the W-algebra $W(2,2)$,   the gap-$p$ Virasoro algebras, the Fermion-Virasoro algebra, the $N=1$ BMS superalgebra, the  planar Galilean conformal algebra. Then we   give the classification of cuspidal $A\fL$-modules by using the weighting functor from $U(\mathfrak{h})$-free modules to weight modules. Consequently, we give the classification of simple cuspidal $\fL$-modules by using the $A$-cover method. Finally, all simple quasi-finite modules over $\fL$ and $\hL$ are classified.  Our results  recover many known Lie superalgebra results from mathematics and mathematical physics, and give many new Lie superalgebras.
\end{abstract}


\end{frontmatter}


\section{Introduction}
We denote by $\bZ, \bZ_+, \bN, \bQ$ and $\bC$ the sets of all integers, non-negative integers, positive integers, rational numbers and complex numbers, respectively. All vector spaces and algebras in this paper are over $\bC$. Any module over a Lie superalgebra or an associative superalgebra is assumed to be $\bZ_2$-graded. A vector space $V$ is called a superspace if $V$ is endowed with a $\bZ_2$-gradation $V=V_{\bar 0}\oplus V_{\bar 1}$. The parity of a homogeneous element $v\in V_{\bar{i}}$ is denoted by $|v|=\bar{i}\in \bZ_2$.  Throughout this paper, $v$ is always assumed to be a homogeneous vector whenever we write $|v|$ for a vector $v\in V$.

Let $A=\bC[t,t^{-1}]$ be the Laurent polynomial algebra. The Witt algebra $W=\Der(A)$ has a basis $\{\fl_i=t^{i+1}\frac{d}{dt}\,|\, n\in \mathbb{Z}\}$ with Lie brackets given by
$$[\fl_i, \fl_j]=(j-i)\fl_{i+j}.$$  The Virasoro algebra  $\Vir=W\oplus   \bC z$ (the universal central extension of the Witt algebra $W$) and the Affine Kac-Moody (super)algebras are two important classes of infinite dimensional Lie (super)algebras that have been studied and used by many mathematicians and physicists in many  different research areas. Weight modules with finite-dimensional weight spaces are called quasi-finite modules (also Harish-Chandra modules, finite modules in literature). Such modules were classified for many Virasoro-related (super)algebras including the Virasoro algebra \cite{Ma}, the higher rank Virasoro algebra \cite{LZ,S2}, the generalized Virasoro algebra \cite{GLZ1}, the twisted Heisenberg-Virasoro algebra \cite{LZ2}, the $W$-algebra $W(2,2)$ \cite{LGZ}, the Schr\"{o}dinger-Virasoro algebra \cite{L}, the non-twisted affine-Virasoro algebra \cite{ GHL,LPX,Rao2}, the gap-$p$ Virasoro algebras \cite{X}, the mirror Heisenberg-Virasoro algebra \cite{LPXZ}, the Fermion-Virasoro algebra \cite{XZ}, the map Virasoro-related (super)algebras \cite{CLW,GLZ2,Rao}. For more related results, we refer the reader to \cite{A,B,CP,Rao,Rao1,GG,HLW,IK,LiuZ,NY,OR} and references therein.

Let $\dg$ be a finite dimensional Lie superalgebra (maybe a Lie algebra), $d$ be an even derivation on $\dg$, and $\sigma$ be an order $n$ automorphism of $\dg$ that commutes with $d$. Then $\dg=\oplus_{i=0}^{n-1}  \dg_{[i]}$, where $\dg_{[i]}=\{g\in \dg|\sigma(g)=\omega_n^i g\}$ for all $[i]=i+n\bZ\in \bZ_n$ and $\omega_n={\rm exp}(\frac{2\pi \sqrt{-1}}{n})$. The automorphism $\tilde{\sigma}$ of the loop algebra $\dg\otimes \bC[t^{\frac 1n},t^{-\frac 1n}]$ is defined by $\tilde{\sigma}(x\otimes t^{\frac kn})=\sigma(x)\otimes (\omega_n^{-1}t^{\frac 1n})^k$.  Let $\fg$ be the fixed point subalgebra of $\tilde{\sigma}$, i.e.,  $$\fg=\oplus_{i=0}^{n-1}\dg_{[i]}\otimes t^{\frac in}A.$$ Then we have the Lie superalgebra $\fL(\dg,d,\sigma)= W\ltimes\fg$ with brackets
\begin{align*}
&[\fl_i,x\otimes t^a]=(ax+id(x))\otimes t^{i+a},\\
&[x\otimes t^a, y\otimes t^b]=[x,y]\otimes t^{a+b},
\end{align*}
where $\fl_i=t^{i+1}\frac{d}{d t}$ for all $i\in \bZ; x\otimes t^a,y\otimes t^b\in \fg$ for all $a,b\in \frac 1n \bZ$. Now we have the natural  $\frac 1n \bZ$-gradation: $$\fL=\fL(\dg,d,\sigma)=\oplus_{a\in \frac 1n \bZ}\fL_a\text{ where }\fL_a=\{x\in \fL:[\fl_0,x]=ax\}.$$

 Throughout this paper, we assume the following technical condition:

 \begin{equation}1 \mbox{ is not an eigenvalue of}\,\, d.\end{equation}

\noindent Then $d-1$ acts bijectively on $\dg$. From $\fL=[\fl_0,\fL]+[\fl_1,\fL_{-1}]$ we know that $\fL$ is perfect.  Certainly $\fL$  is perfect if $\dg$ is perfect even if Condition (1.1) does not hold. But there are many existing  examples for $\fL$ to be perfect while $\dg$ is not.

We call the Lie superalgebra  $\fL$ and its   universal central extension  $\hL=\hL(\dg,d,\sigma)$   as $(d,\sigma)$-\emph{twisted Affine-Virasoro superalgebras}. There exist many known and new interesting examples of such Lie (super)algebras in the literature, including the    Affine-Virasoro superalgebras,  the twisted Heisenberg-Virasoro algebra, the mirror Heisenberg-Virasoro algebra, the W-algebra $W(2,2)$,   the gap-p Virasoro algebras, the Fermion-Virasoro algebra, the $N=1$ BMS superalgebra, the  planar Galilean conformal algebra. Certainly there are also many new interesting $(d,\sigma)$-twisted Affine-Virasoro superalgebras.
 See examples in Section 5.

 Let $L$ be any subalgebra of $\fL$ or $\hL$ containing $\fl_0$. A $L$-module $M$ is called a \emph{weight} module if the action of $\fl_0$ on $M$ is diagonalizable i.e., $M=\bigoplus\limits_{\lambda\in\bC}M_\lambda$, where $M_\lambda=\{v\in M\,|\,\fl_0v=\lambda v\}$ is called the weight space of weight $\lambda$. The set $$\Supp(M):=\{\lambda\in\bC\,|\,M_\lambda\neq0\}$$ is called the \emph{support} of $M$.  Clearly, if $M$ is a simple weight module, then $\Supp(M)\subseteq\lambda+\frac{1}{n}\bZ$ for some $\lambda\in\bC$. A weight module $M$ is called \emph{quasi-finite} if all its weight spaces are finite-dimensional.  A quasi-finite weight module is called \emph{cuspidal} (\emph{uniformly bounded}) if the dimensions of its weight spaces are uniformly bounded, i.e., there exists $N\in\bN$ such that $\dim M_\lambda\le N$ for all $\lambda\in\Supp(M)$.

 In this paper, we first determine the universal central extension $\hL=\hL(\dg,d,\sigma)$, give the classification of cuspidal $A\fL$-modules by using the weighting functor from $U(\mathfrak{h})$-free modules to weight modules. Then we give the classification of simple cuspidal $\fL$-modules by using the $A$-cover method. Finally, all simple quasi-finite modules over $\fL$ and $\hL$ are classified.

The main results of this paper are as follows (for the notations see Sections 2 and 3).

\begin{theo}\label{thm1} Let $\dg$ be a finite dimensional Lie (super)algebra, $d$ be an even derivation on $\dg$ without eigenvalue $1$, and $\sigma$ be an order $n$ automorphism of $\dg$ that commutes with $d$. Then $$\aligned H^2(\fL(\dg,d,\sigma),\bC^{1|1})_{\bar 0}\cong & H^2(\dg,\bC^{1|1})_{\bar 0}^{d,\sigma}\oplus {(\Inv(\ddot{\fg}))}^\sigma\\ &\oplus \Big(\dg/\big((d+1)\dg+[(d+\frac 12)\dg,\dg]+ [\dg,[\dg,\dg]]\big)\Big)^{\sigma}.\endaligned$$\end{theo}


  \begin{theo}\label{thm2} Let $\dg$ be a finite dimensional Lie (super)algebra, $d$ be an even derivation on $\dg$ without eigenvalue $1$, and $\sigma$ be an order $n$ automorphism of $\dg$ that commutes with $d$. Then any simple quasi-finite $\hL(\dg,d,\sigma)$-module is  a highest weight module, a lowest weight module, or isomorphic to a simple sub-quotient of a loop module $\Gamma(V,\lambda)$ for some $\lambda\in \bC$ and a finite dimensional simple $\ddg$ module $V$.
   \end{theo}

The paper is structured as follows. In Section 2, we  determine  the universal central extension $\hL=\hL(\dg,d,\sigma)$ for the newly introduced  $(d,\sigma)$-twisted Affine-Virasoro superalgebra $\fL=\fL(\dg,d,\sigma)$, i.e., give the proof of Theorem \ref{thm1} and obtain the universal central extension $\hL$ of $\fL$.  
In Section 3, we   give the classification of simple cuspidal $A\fL$-modules  by using the weighting functor from $U(\mathfrak{h})$-free modules to weight modules.  In Section 4, by using the $A$-cover method  we   give the proof of Theorem \ref{thm2}, i.e., classify all simple   quasi-finite modules over  $\fL$ and $\hL$. These  theorems are generally   easy to apply in many cases. In Section 5, we   give some concrete examples with various $\dg$  which  we  recover and generalize    many known results. Certainly our $(d,\sigma)$-twisted Affine-Virasoro superalgebras $\hL(\dg,d,\sigma)$ will give a lot of new Lie (super)algebras in this class. We remark that  from our   Affine-Virasoro superalgebras $\hL=\hL(\dg,d,$ id) we can construct some interesting vertex operator algebras and superalgebras \cite{L1, L2}.

 \begin{section}{Universal central extensions of $\fL$}\end{section}

 In this section, we will give the proof of Theorem \ref{thm1}, i.e., determine the universal central extension $\hL$ for $\fL$. We first set up the notations.

 Let $L$ be any Lie superalgebra. A $2$-cocycle  $\alpha: L\times L\rightarrow \bC $  is a bilinear form satisfying
\begin{align*}
&\alpha(x,y)=-(-1)^{|x||y|}\alpha(y,x),\\
&\alpha(x,[y,z])=\alpha([x,y],z)+(-1)^{|x||y|}\alpha(y,[x,z]),\forall x,y,z\in L,
\end{align*}
and it is called a $2$-coboundary if there is a linear map $f:L\rightarrow \bC$ with $\alpha(x,y)=f([x,y])$ for all $x,y\in L$. Define
$$\alpha_{\bar0}(x,y)=\left\{ \begin{array}{cl}
	\alpha(x,y), &\text{ if } |x|+|y|=\bar 0, \\
		0, &\text{ if } |x|+|y|=\bar 1,\end{array} \right.$$
		
		$$\alpha_{\bar 1}(x,y)=\left\{ \begin{array}{cl}
			0, &\text{ if } |x|+|y|=\bar 0, \\
			\alpha(x,y), &\text{ if } |x|+|y|=\bar 1,
\end{array} \right.$$ that is, the element $\alpha_{\bar 0}(x,y)$ has even parity and $\alpha_{\bar 1}(x,y)$ has odd parity.

Then there is a 1-1 correspondence between $2$-cocycle $\alpha$  and the central extension $(L\oplus \bC^{1|1},[,]')$ of $L$ with brackets
$$[x,y]'=[x,y]+\big(\alpha_{\bar 0}(x,y),\alpha_{\bar 1}(x,y)\big).$$



Denote by $\Bil(L), Z^2(L)$, and $B^2(L)$  the vector space consists of all bilinear forms, $2$-cocyles, $2$-coboundaries   on $L$, respectively. Denote the set of  supersymmetric superinvariant bilinear forms on $L$ as
$$ \Inv(L)=\{   \alpha\in \Bil(L): \alpha(x,y)=(-1)^{|x||y|}\alpha(y,x)\text{ and }\alpha([x,y],z)=\alpha(x,[y,z])\}.$$

 Then we have $H^2(L,\bC^{1|1})_{\bar 0}\cong Z^2(L)/B^2(L)$ (see for example Section 16.4 in \cite{Mu}).

Let $\jmath$ be any linear operator on a vector space $V$ such that $f(\jmath)=0$ for some $ f(t)\in \bC[t]\setminus\{0\}$. Then $V$ has a generalized  eigenspace decomposition $V=\oplus_{\lambda: f(\lambda)=0} V_{(\lambda)}$, where   $V_{(\lambda)}$ is the generalized  eigenspace of $\jmath$ with respect to the eigenvalue $\lambda$, i.e.,
the subspace consists of all $v\in V$   annihilated by some powers of $\jmath-\lambda$. And denote  $V^{\jmath}:=\{v\in V|\jmath v=v\}$.

Let $d$ be any even derivation of $L$, such that $f(d)=0$ for some $ f(t)\in \bC[t]\setminus\{0\}$. We have a linear operator, which we still denote by $d$, on $\Bil(L)$ defined by
$$ (d\alpha)(x,y)=\alpha(dx,y)+\alpha(x,dy).$$
Then $\Bil(L)_{(\lambda)}=\{\alpha\in \Bil(L)|\alpha(L_{(\mu)},L_{(\nu)})=0,\forall \mu,\nu\in \bC, \mu+\nu\ne \lambda\},$
where the generalized eigenspaces are with respect to $d$.

Clearly that $B^2(L)$ and $Z^2(L)$ are $d$-invariant, and  they have generalized  eigenspace decompositions with only finite many nonzero generalized  eigenspaces. So is $H^2(L)$.

For any automorphism $\tau$ of $L$, we have an automorphism, which we still denote by $\tau$, of $\Bil(L)$ defined by
$$(\tau \alpha)(x,y)=\alpha(\tau x,\tau y),\forall x,y\in L,\forall \alpha\in B(L).$$

Clearly, $B^2(L),Z^2(L), \Inv(L)$ are $\sigma$-invariant. Let \begin{equation}\ddot{\fg}=\bC \partial\ltimes \dg\end{equation} be the Lie superalgebra of dimension $1+\dim \dg$ with brackets $[\partial,x]=d(x),\forall x\in \dg$. The automorphism $\sigma$ is naturally extended to $\ddot{\fg}$ by $\sigma(\partial)=\partial$.

Now we are going to prove  Theorem \ref{thm1} via eight auxiliary  lemmas.

\

\begin{lemm}\label{zero}Let $V$ be a finite dimensional vector space over $\bC$, $\jmath\in \gl(V)$, $V_{j}$ be a $\jmath$-invariant subspace of the generalized  eigenspace $V_{(\lambda_j)}$ for $j=1,2$,  $B: V_{1}\times V_{2}\rightarrow \bC $ be a bilinear map, and $f_i(t), g_i(t)\in \bC[t],i=1,2,\ldots,s$. If $\sum_{i=1}^s f_i(\lambda_1)g_i(\lambda_2)\ne 0$ and  $\sum_{i=1}^s B(f_i(\jmath)v_1,g_i(\jmath)v_2)=0$ for all $v_j\in V_j$, then $B=0$.\end{lemm}

\begin{proof}Let $W_1$ be the maximal $\jmath$-invariant subspace of $V_1$ with $B(W_1,V_2)=0$. If $V_1\ne W_1$, we may choose $v\in V_1\backslash W_1$ such that $(\jmath-\lambda_1)v\in W_1$. Then $$0=\sum_{i=1}^sB(f_i(\jmath)v,g_i(\jmath)v_2)=\sum_{i=1}^sB(f_i(\lambda_1)v,g_i(\jmath)v_2)= B\big(v,\sum_{i=1}^s f_i(\lambda_1)g_i(\jmath)v_2\big),\forall v_2\in V_2.$$ From $\sum_{i=1}^s f_i(\lambda_1)g_i(\lambda_2)\ne 0$, we know that  $\sum_{i=1}^s f_i(\lambda_1)g_i(\jmath)$ acts injectively hence bijectively on $V_2$. Then $B(v,V_2)=0$.  Now $W'=\bC v+ W_1$ is a $\jmath$-invariant subspace of $V_1$ with  $B(W', V_2)=0,\forall v_1\in W',v_2\in V_2$, a contradiction. Thus $W=V_1$ and $B=0$.\end{proof}

From now on, let $\dg$ be a finite dimensional Lie (super)algebra, $d$ be an even derivation on $\dg$, and $\sigma$ be an order $n$ automorphism of $\dg$ that commutes with $d$, and $\fL=\fL(\dg,d,\sigma)$.

Denote $x\otimes t^a\in\fL$ by $xt^a$ or $x(a)$ for short for any $a\in \frac 1n \bZ$. We extend $d,\sigma$ to a derivation and an automorphism, which we still denote by $d,\sigma$, of $\fL$ by
\begin{equation}\label{d-sigma}
\aligned d(\fl_i)=&0,  \,\,\,d(x t^a)=d(x) t^a;\\ \sigma(\fl_i)=&\fl_i, \,\,\,\sigma(xt^a)=\sigma(x)t^a.\endaligned\end{equation}

Let $H:=\left\{\alpha\in Z^2(\fL)|\alpha(\fl_1,\fL_{-1})=\alpha(\fl_0,\fL_a)=0, \forall a\in \frac 1n \bZ\setminus\{ 0\}\right\}$. Then we have the generalized eigenspace $H_{(\mu)}$ for any $ \mu\in \bC$ with respect to $d$. We will study properties of the space $H$ in the next three lemmas.

\begin{lemm}\label{lemm-5} \begin{itemize}
		\item [(a).] $H$ is a super subspace of $Z^2(\fL)$ that is $d$-invariant and $\sigma$-invariant.
		\item [(b).] $Z^2(\fL)=H\oplus B^2(\fL)$. \end{itemize}\end{lemm}

\begin{proof}  (a). This is easy to verify.
	
	(b). For any $\alpha\in H\cap B^2(\fL)$,  there exists $f\in \fL^*$ such that $\alpha(u,v)=f([u,v]),\forall u,v\in \fL$. Then $$\aligned &f(\fl_0)=-\frac 12f([\fl_1,\fl_{-1}])=-\frac 12\alpha(\fl_1,\fl_{-1})=0,\\
&	f(xt^0)=f([\fl_1,(d-1)^{-1}xt^{-1}])=\alpha\big(\fl_1,(d-1)^{-1}xt^{-1}\big)=0,\\
	&f(u)=\frac 1a f([\fl_0,u])=\frac 1a \alpha(\fl_0,u)=0,\forall u\in \fL_{a}\text{ with }a\ne 0.\endaligned$$
	So $f=0$. Thus $\alpha=0$, that is, the sum of $H$ and $B^2(\fL)$ is direct. Now for any $\alpha\in Z^2(\fL)$, we have $f\in \fL^*$ defined by
	$$\aligned &f(\fl_0)=-\frac 12\alpha(\fl_1,\fl_{-1}),\\ &f(xt^0)=\alpha(\fl_1,(d-1)^{-1}xt^{-1}),\\
	 &f(u)=\frac 1a \alpha(\fl_0,u),\forall u\in \fL_a\text{ with }a\ne 0.\endaligned $$
	Then we have $\beta\in B^2(\fL)$ defined by $\beta(u,w)=f([u,w])$, and $\alpha-\beta\in H$, i.e., $Z^2(\fL)=H+B^2(\fL)$. So we have proved that $Z^2(\fL)=H\oplus B^2(\fL)$.
\end{proof}

\begin{lemm} \label{lemm-6} Let  $\alpha\in H, i\in \bZ$.
	\begin{itemize}
		\item[(a).]
 We have \begin{equation}\label{d-g-4}\alpha(\fl_{i},xt^{-i})=\alpha\left(\fl_2,\Big((\binom i2-\binom {i+1}3)d+\binom i2\Big)xt^{-2}\right).\end{equation}
 	\item[(b).]
$ \alpha(\fl_2,(d^2+d)(x)t^{-2})=0,\forall x\in \dg.$
 	\item[(c).]   $\alpha(\fl_{i},xt^{-i})=0$ if $xt^{-i}\in \fL_{(\mu)}$ with $\mu\ne0, -1$.
 		\item[(d).] $\alpha(\fL_a,\fL_b)=0,$ if $ a+b\ne 0$.
 		\end{itemize}\end{lemm}

\begin{proof} (a). From $\alpha([\fl_i,\fl_j],xt^{-i-j})=\alpha([\fl_i,x t^{-i-j}],\fl_j)+\alpha(\fl_i,[\fl_j,x t^{-i-j}])$, we have
 \begin{equation}\label{d-g-1} (j-i)\alpha(\fl_{i+j},x t^{-i-j})=\alpha(\fl_j,(i+j-id)x t^{-j})+\alpha(\fl_i,(jd-i-j)(x) t^{-i}).\end{equation}

Taking $j=1$ in (\ref{d-g-1}), we have $(1-i)\alpha(\fl_{i+1},xt^{-i-1})=\alpha(\fl_i,(d-i-1)(x)t^{-i})$, i.e., $\alpha(\fl_{i},xt^{-i})=-\alpha(\fl_{i-1},\frac{d-i}{i-2}(x)t^{1-i})$. Then
\begin{equation}\label{d-g-2}\alpha(\fl_{i},xt^{-i})=(-1)^{i-2}\alpha(\fl_2,\binom{d-3}{i-2}xt^{-2}),\forall i\ge 2,\end{equation}
where $\binom{t}i:=\frac{t(t-1)\cdots(t-i+1)}{i!}, \binom{t}{0}:=1$.

Taking $j=2,i=3$ in (\ref{d-g-1}), together with (\ref{d-g-2}), we deduce that  $$-(-1)^5\alpha(\fl_2, \binom{d-3}{3}(x)t^{-2})=\alpha(\fl_2,(5-3d)(x)t^{-2})+(-1)^3 \alpha(\fl_2,\binom{d-3}{1}(2d-5)(x)t^{-2}),$$ i.e., $\alpha(\fl_2,(d^3-d)(x)t^{-2})=0$.
Since we have assumed that $1$ is not an eigenvalue, we get \begin{equation}\label{d-g-3}\alpha(\fl_2,(d^2+d)(x)t^{-2})=0,\forall x\in \dg.\end{equation}

Note that $(-1)^{i-2}\binom{d-3}{i-2}\equiv (\binom i2-\binom {i+1}3)d+\binom i2 \,\,\pmod{  d^2+d},\forall i\ge 2$.

From (\ref{d-g-2}) and (\ref{d-g-3}), we have (\ref{d-g-4}) hold for any $i\ge 2$. It is clear that (\ref{d-g-4}) holds for $i=1$.


Now for any  $j\le 0$, we may choose $i$ such that $i+j\ge 2$ and $1+\frac ji$ is not an eigenvalue of $d$, then from (\ref{d-g-1},\ref{d-g-3}), we get
\begin{equation*}\begin{aligned}\alpha(&\fl_j,(i+j-id)(x) t^{-j})\\ &-\alpha\Big (\fl_2,((\binom j2-\binom {j+1}3)d+\binom j2)(i+j-id)(x)t^{-2}\Big)\\ =&(j-i)\alpha\big(\fl_{i+j},x t^{-i-j})-\alpha(\fl_i,(jd-i-j)(x) t^{-i}\big)\\&-\alpha\big(\fl_2,((\binom j2-\binom {j+1}3)d+\binom j2)(i+j-id)(x)t^{-2}\big)\\
=&(j-i)\alpha\big(\fl_2,((\binom {i+j}2-\binom {i+j+1}3)d+\binom {i+j}2)(x)t^{-2}\big)\\ &+\alpha\big(\fl_2,((\binom i2-\binom {i+1}3)d+\binom i2)(i+j-jd)(x)t^{-2}\big)\\
 &-\alpha\big(\fl_2,((\binom j2-\binom {j+1}3)d+\binom j2)(i+j-id)(x)t^{-2}\big)\\
 =&0. \end{aligned}\end{equation*}
 Therefore (\ref{d-g-4}) holds for all $i$.

 (b). This is (\ref{d-g-3}).

 (c). This follows from (\ref{d-g-4}) and (\ref{d-g-3}).

 (d). For any $\alpha\in H$,   $a,b\in \frac 1n \bZ$ with $a+b\ne0$, from $$0=\alpha(\fl_0,[u,v])=\alpha([\fl_0,u],v)+\alpha(u,[\fl_0,v])=(a+b)\alpha(u,v)\forall u\in \fL_a,v\in \fL_b,$$ we know that $\alpha(\fL_a,\fL_b)=0,$ if $ a+b\ne 0$.
 \end{proof}

Let  $\alpha\in H$. From $\alpha(\fl_{k},[xt^{a-k},yt^{-a}])=\alpha([\fl_{k},xt^{a-k}],yt^{-a})+\alpha(xt^{a-k},[\fl_{k},yt^{-a}])$, we deduce that
\begin{equation}\label{g-g-1}\alpha((a+k(d-1))xt^a,yt^{-a})=\alpha(xt^{a-k},(a-kd)yt^{k-a})+\alpha(\fl_{k},[x,y]t^{-k}).\end{equation}
Replacing $k$ with $2k$, we have $$\alpha((a+2k(d-1))xt^a,yt^{-a})=\alpha(xt^{a-2k},(a-2kd)(y)t^{2k-a})+\alpha(\fl_{2k},[x,y]t^{-2k}).$$
Replacing $y$ with $(a-k-kd)(a-kd)y$, we get
\begin{equation}\label{eq-1}\begin{aligned}\alpha\big((a+2k(d-1))&xt^a,(a-k-kd)(a-kd)yt^{-a}\big)\\ =&\alpha\big(xt^{a-2k},(a-k-kd)(a-kd)(a-2kd)(y)t^{2k-a}\big)\\
		&+\alpha\big(\fl_{2k},[x,(a-k-kd)(a-kd)y]t^{-2k}\big).\end{aligned}\end{equation}
Substituting $x$ with $(a+k(d-2)) x$ in
 $(\ref{g-g-1})$ and then using $(\ref{g-g-1})$ with $a$ replaced by $a-k$ and $y$ replaced by $(a-kd)y$, we have $$\aligned \alpha((a+&k(d-2))(a+k(d-1))xt^a,yt^{-a})\\=&\alpha((a+k(d-2))xt^{a-k},(a-kd)yt^{k-a})+\alpha(\fl_{k},[(a+k(d-2))x,y]t^{-k})\\
=&\alpha(xt^{a-2k},(a-k-kd)(a-kd)(y)t^{2k-a})+\alpha(\fl_{k},[x,(a-kd)y]t^{-k})\\
&+\alpha(\fl_{k},[(a+k(d-2))x,y]t^{-k}).\endaligned $$ Replacing $y$ with $(a-2kd)y,$ we obtain
\begin{equation}\label{eq-2}\begin{aligned}\alpha((a&+k(d-2))(a+k(d-1))xt^a,(a-2kd)yt^{-a})\\
=&\alpha(xt^{a-2k},(a-k-kd)(a-kd)(a-2kd)yt^{2k-a})\\&+\alpha(\fl_{k},[x,(a-2kd)(a-kd)y]t^{-k})+\alpha(\fl_{k},[(a+k(d-2))x,(a-2kd)y]t^{-k}).\end{aligned}\end{equation}
Equation (\ref{eq-1}) minus Equation (\ref{eq-2}) gives
\begin{equation}\label{o-m}\begin{aligned}\alpha((&a+2k(d-1))xt^a,(a-k-kd)(a-kd)yt^{-a})\\& -\alpha((a+k(d-2))(a+k(d-1))xt^a,(a-2kd)yt^{-a})\\
=&-\alpha(\fl_{k},[x,(a-2kd)(a-kd)y]t^{-k}) -\alpha(\fl_{k},[(a+k(d-2))x,(a-2kd)y]t^{-k})
\\&+\alpha(\fl_{2k},[x,(a-k-kd)(a-kd)y]t^{-2k}).\end{aligned}\end{equation}
The coefficient of $k^3$ in the left hand side of equation (\ref{o-m}) is
\begin{equation}\label{lhd}\begin{aligned} \alpha(2(d-1)&xt^a, d(d+1)yt^{-a})-\alpha((d-2)(d-1)xt^a,-2dyt^{-a})\\
&=2(\alpha((d-1)xt^a, (d+1)dyt^{-a})+\alpha((d-2)(d-1)xt^a,dyt^{-a}))\\
&=2\Big(\alpha((d-1)xt^a, (d-\frac 12)dyt^{-a})+\alpha((d-\frac 12)(d-1)xt^a,dyt^{-a})\Big).\end{aligned}\end{equation}





\begin{lemm}\label{lemm-7} We have the following vector space decomposition $H=H_{(-1)}\oplus H_{(0)}\oplus H_{(1)}$. \end{lemm}

\begin{proof} For any $\alpha\in H$, write $\alpha=\sum_{\mu \in \bC} \alpha_\mu$ with $\alpha_\mu\in H_{(\mu)}$. Applying (\ref{o-m}) to $\alpha_\mu$ with $\mu \ne 0,1,-1$, and using Lemma \ref{lemm-6}(b) we have $$\begin{aligned} \alpha_\mu&\big((a+2k(d-1))xt^a,(a-k-kd)(a-kd)yt^{-a}\big)\\ -&\alpha_\mu\big((a+k(d-2))(a+k(d-1))xt^a,(a-2kd)yt^{-a}\big)=0,  \forall k.\end{aligned}$$
So the coefficient of $k^3$ is zero. From (\ref{lhd}), we see that  $$\alpha_\mu\big((d-1)xt^a, (d-\frac 12)dyt^{-a}\big)+\alpha_\mu\big((d-\frac 12)(d-1)xt^a,dyt^{-a}\big)=0.$$
Then $\alpha_\mu\big(xt^a, (d-\frac 12)dyt^{-a}\big)+\alpha_\mu((d-\frac 12)xt^a,dyt^{-a})=0$.

Now for any $\lambda\ne 0$, Applying Lemma \ref{zero} to
$$\begin{aligned} B:&(\dg_{[na]})_{(\mu-\lambda)}\times (\dg_{[-na]})_{(\lambda)}\rightarrow \bC,\\  &(x,y)\mapsto \alpha_\mu(xt^a,yt^{-a}),\end{aligned}$$  we know that $B=0$,  i.e., $\alpha_\mu(\fg_{(\mu-\lambda)},\fg_{(\lambda)})=0$. And from $\alpha_\mu(\fg_{(0)},\fg_{(\mu)})=0$ we know that  $\alpha_\mu(\fg_{(\mu)},\fg_{(0)})=0$. So we have proved $\alpha_\mu=0,$ for all $\mu\ne 0,1,-1$. Hence $H=H_{(-1)}\oplus H_{(0)}\oplus H_{(1)}$. \end{proof}

We will completely determine  the spaces $H_{(-1)}, H_{(0)}, H_{(1)}$ in the next four lemmas.

\begin{lemm} \label{lemm-8}We have the following vector space monomorphisms:
\begin{itemize}
	\item[(1).]
 $\pi_{-1}: \left (\Big (\dg\Big/\left ((d+1)\dg+[(d+\frac 12)\dg,\dg]+[\dg,[\dg,\dg]]\Big)\right)^{\sigma}\right)^*\rightarrow H_{(-1)}$ defined by
\begin{equation}\label{Eq2.12}\begin{aligned}&(\pi_{-1}f)(\fl_i,\fl_j)=0, \,\,(\pi_{-1}f)(\fl_i,xt^a)=\delta_{i+a,0}\frac{i^3-i}6 f(x),\\ &(\pi_{-1}f)(xt^a,yt^{b})=\delta_{a+b,0}\frac{1-4a^2}{12}f([x,y]),\forall i,j\in \bZ, xt^a,yt^b\in \fg;\end{aligned}\end{equation}

	\item[(2).]$\pi_0: {(\Inv(\ddot{\fg}))}^\sigma \rightarrow H_{(0)}$ defined by
\begin{equation}\label{Eq2.13}\begin{aligned}&(\pi_0B)(\fl_i,\fl_j)=\delta_{i+j,0}\frac{i^3-i}{12} B(\partial,\partial),\,\,(\pi_0 B)(\fl_i,xt^a)=\delta_{i+a,0}(i^2-i) B(\partial,x),\\ &(\pi_0B)(xt^a,yt^b)=\delta_{a+b,0}\big(aB(x,y)+B(\partial,[x,y])\big),\forall i,j\in \bZ, xt^a,yt^b\in \fg;\end{aligned}\end{equation}

	\item[(3).] $\pi_{1}:(Z^2(\dg))^{d,\sigma} \rightarrow H_{(1)}$ defined by
\begin{equation}\label{Eq2.14}\begin{aligned}&(\pi_{1}\dot\alpha)(\fl_i,xt^j)=(\pi_1\dot\alpha)(\fl_i,\fl_j)=0,\\& (\pi_1\dot\alpha)(xt^a,yt^b)=\delta_{a+b,0}\dot\alpha(x,y),\forall i,j\in \bZ, xt^a,yt^b\in \fg.\end{aligned}\end{equation}\end{itemize}
\end{lemm}

\begin{proof} (1).
	Take any $f\in  \left (\Big (\dg\Big/\left ((d+1)\dg+[(d+\frac 12)\dg,\dg]+[\dg,[\dg,\dg]]\Big)\right)^{\sigma}\right)^*$. Note that
	$$f(dx)=-f(x), f([dx, y])=-\frac12  f([x, y]), f([[\dg, \dg],\dg])=0, \forall x,y\in\dg.
	$$

 We first verify that $\rho:=\pi_{-1}(f)\in Z^2(\fL)$.
For any $x,y,z\in\dg, a,b\in \frac 1n\bZ$, we compute
\[
\rho([xt^a,yt^{b}], zt_{-a-b})+\rho([yt^{b}, zt_{-a-b}], xt^a,)+\rho([ zt_{-a-b},xt^a],yt^{b})=0;
\]
\[\begin{aligned}
\rho\big(\fl_k,&[xt^a,yt^{-a-k})-\rho([\fl_k,xt^a],yt^{-a-k})-\rho(xt^a,[\fl_k,yt^{-a-k}\big)\\
=&\rho(\fl_{k},[xt^a, yt^{-a-k}])-\rho((a+k d) x t^{a+k}, y t^{-a-k})-\rho(x t^a,(-a-k+k d) y t^{-a}) \\
=&\frac{f([x, y])}{6}(k^{3}-k)+\frac{f([x, y])}{3} a((a+k)^{2}-\frac{1}{4})+\frac{f([d x, y])}{3} k((a+k)^{2}-\frac{1}{4}) \\
&-\frac{f([x, y])}{3}(a+k)(a^{2}-\frac{1}{4})+\frac{f([x, d y])}{3} k(a^{2}-\frac{1}{4}) \\
=&\frac{f([x, y])}{6}(k^{3}-k)+\frac{f([x, y])}{3}a((a+k)^{2}-\frac{1}{4})-\frac{f([x, y])}{6} k((a+k)^{2}-\frac{1}{4}) \\
&-\frac{f([x, y])}{3}(a+k)(a^{2}-\frac{1}{4})-\frac{f([x, y])}{6} k(a^{2}-\frac{1}{4}) \\
=&0;
\end{aligned}\]

$\begin{aligned}\rho\big(\fl_{i}&,\left[\fl_{j}, x t^{-i-j}\right]\big)-\rho\big(\left[\fl_{i}, \fl_{j}\right], x t^{-i-j}\big)-\rho\big(\fl_{j},\left[\fl_{i}, x t^{-i-j}\right]\big) \\ =&\frac{f(x)}{6}(i^{3}-i)(-i-j)+\frac{f(d x)}{6}(i^{3}-i) j-\frac{f(x)}{6}((i+j)^{3}-(i+j))(j-i) \\ &+\frac{f(x)}{6}(j^{3}-j)(i+j)-\frac{f(d x)}{6}(j^{3}-j) i \\ =&\frac{f(x)}{6}(i^{3}-i)(-i-j)-\frac{f(x)}{6}(i^{3}-i) j-\frac{f(x)}{6}((i+j)^{3}-(i+j))(j-i) \\ &+\frac{f(x)}{6}(j^{3}-j)(i+j)+\frac{f(x)}{6}(j^{3}-j) i \\   =&0.\end{aligned}$

The fact $\rho\in H_{(1)}$ follows from the definition.

Similarly we can verify (2) and (3) (simpler than (1)).  
 \end{proof}



 \begin{lemm}\label{lemm-9} The linear map $\pi_{-1}$ defined in (\ref{Eq2.12}) is surjective. \end{lemm} 

 \begin{proof} Take $\alpha_{-1}\in H_{(-1)}$. For any $xt^a\in \dg_{(\mu)},yt^{b}\in \dg_{(\lambda)}
 	$ (the generalized eigenspaces are with respect to $d$ defined in  (\ref{d-sigma})) where $a, b\in\frac 1n\bZ, \lambda,\mu\in\bC$, we have
 	$$\alpha_{-1}(xt^a,yt^b)=0\text{ if }a+b\ne0\text{ or }\lambda+\mu\ne -1.$$
 	So we take  $xt^a\in \dg_{(\mu)},yt^{-a}\in \dg_{(-1-\mu)}$. Formulas we will get actually hold for any $x, y$ even they are not generalized eigenvectors with respect to $d$. Since
 	$d\dg_{(-1)}=\dg_{(-1)}$,
 	using Lemma \ref{lemm-6}(b) we know   that $\alpha_{-1}(\fl_2  ,(d+1)\dg_{(-1)})=0$. Furthermore
 	\begin{equation}\label{d+1}
 	\alpha_{-1}(\fl_2  ,(d+1)\dg)=0.\end{equation}
 	Applying this to  (\ref{o-m}), we have
 \begin{equation}\label{o-m--1}\begin{aligned}\alpha_{-1}\big((&a+2k(d-1))xt^a,(a-k-kd)(a-kd)yt^{-a}\big)\\ &-\alpha_{-1}\big((a+k(d-2))(a+k(d-1))xt^a,(a-2kd)yt^{-a}\big)\\ =&-\binom {k+1}3\alpha_{-1}\big(\fl_{2},[x,(a-2kd)(a-kd)y]t^{-2}\big)\\ & -\binom {k+1}3\alpha_{-1}\big(\fl_{2},[(a+k(d-2))x,(a-2kd)y]t^{-2}\big)\\
&+\binom {2k+1}3\alpha_{-1}(\fl_{2},[x,(a-k-kd)(a-kd)y]t^{-2}).\end{aligned}\end{equation}

The coefficients of $k^5$ give \begin{equation}\begin{aligned}&\alpha_{-1}\big(\fl_2, (-[(d-2)x,dy]+[x,d^2y]-4[x,d(d+1)y])t^{-2}\big)\\
&=\alpha_{-1}\big(\fl_2, ((-d-2)[x,dy]-2[x,d^2y])t^{-2}\big)\\ &=-2\alpha_{-1}\big(\fl_2,[x, (d^2+\frac d2)y]t^{-2}\big)=0.\end{aligned}\end{equation}

If $y\not\in \dg_{(0)}$, i.e., $\mu\ne -1$,   we have $\alpha_{-1}(\fl_2,[x, (d+\frac 12)y]t^{-2}])=0$, i.e., $$\alpha_{-1}(\fl_2,[x, dy]t^{-2}])=-\frac 12 \alpha_{-1}(\fl_2,[x, y]t^{-2}]).$$ Now from $$\alpha_{-1}(\fl_2,[dx,y]t^{-2})=\alpha_{-1}(\fl_2,d[x,y]t^{-2})-\alpha_{-1}(\fl_2,[x,dy]t^{-2})=-\frac 12\alpha_{-1}(\fl_2,[x, y]t^{-2}].$$ Exchanging $x$ and $y$ if $y\in \dg_{(0)}$, we get \begin{equation}\label{2.16}\alpha_{-1}(\fl_2,[x, dy]t^{-2}])=-\frac 12 \alpha_{-1}(\fl_2,[x, y]t^{-2}]),\forall xt^a,yt^{-a}\in \fg,\end{equation}
where we do not need $x, y$ to be generalized eigenvectors with respect to $d$.

Then (\ref{o-m--1}) becomes
$$\begin{aligned}\alpha_{-1}\big((&a+2k(d-1))xt^a,(a-k-kd)(a-kd)yt^{-a}\big)\\&-\alpha_{-1}\big((a+k(d-2))(a+k(d-1))xt^a,(a-2kd)yt^{-a}\big)\\=&-\binom {k+1}3\alpha_{-1}(\fl_{2},[x,(a+k)(a+\frac k2)y]t^{-2})\\& -\binom {k+1}3\alpha_{-1}(\fl_{2},[(a-\frac 52 k)x,(a+k)y]t^{-2})\\
	&
+\binom {2k+1}3\alpha_{-1}(\fl_{2},[x,(a-\frac k2)(a+\frac k2)y]t^{-2})\\
=&(a^2-\frac 14)k^3\alpha_{-1}(\fl_2,[x,y]t^{-2}).\end{aligned}$$

From (\ref{lhd}) the coefficient of $k^3$ gives $$\begin{aligned} 2\Big(\alpha_{-1}((d-1)&xt^a, (d-\frac 12)dyt^{-a})+\alpha_{-1}((d-\frac 12)(d-1)xt^a,dyt^a\Big)\\ =&(a^2-\frac 14)\alpha_{-1}(\fl_2,[x,y]t^{-2})=\frac 43 (a^2-\frac 14)\alpha_{-1}(\fl_2,[(d-1)x,dy]t^{-2}).\end{aligned}$$

Thus $\alpha_{-1}((d-\frac 12)xt^a,yt^{-a})+\alpha_{-1}(xt^a,(d-\frac 12)yt^{-a})=\frac 23(a^2-\frac 14)\alpha_{-1}(\fl_2,[x,y]t^{-2})$.

Let $B(x,y)=\alpha_{-1}(xt^a,yt^{-a})+\frac 13(a^2-\frac 14)\alpha_{-1}(\fl_2,[x,y]t^{-2})$. Then $$B((d-\frac 12)x,y)+B(x,(d-\frac 12)y)=0$$ and Lemma \ref{zero} imply $B=0$, i.e.,
\begin{equation}\alpha_{-1}(xt^a,yt^{-a})=-\frac 13(a^2-\frac 14)\alpha_{-1}(\fl_2,[x,y]t^{-2}).\end{equation}

Finally, from $$\begin{aligned}&\alpha_{-1}\big(xt^a,[yt^{b+k},zt^{-a-b-k}]\big)\\
	=&\alpha_{-1}\big([x,y]t^{a+b+k},zt^{-a-b-k}\big)+(-1)^{|x||y|}\alpha_{-1}\big(yt^{b+k},[x,z]t^{-b-k}\big),\end{aligned}$$ we have
$$\begin{aligned}-\frac 13(a^2-\frac 14)\alpha_{-1}(\fl_2,&[x,[y,z]]t^{-2})=-\frac 13((a+b+k)^2-\frac 14)\alpha_{-1}(\fl_2,[[x,y],z]]t^{-2})\\ &-(-1)^{|x||y|}\frac 13((b+k)^2-\frac 14)\alpha_{-1}(\fl_2,[x,[y,z]]t^{-2}),\forall k\in \bZ.\end{aligned}$$

The coefficient of $k^2$ gives $\alpha_{-1}(\fl_2,[[x,y],z]t^{-2})+(-1)^{|x||y|}\alpha_{-1}(\fl_2,[y,[x,z]]t^{-2})=0$. So \begin{equation}\label{2.18}\alpha_{-1}(\fl_2,[x,[y,z]]t^{-2})=0.\end{equation}

Let $f(x)=\left\{\begin{array}{ll}\alpha_{-1}(\fl_2, xt^{-2}),&\forall x\in \dg_{[0]}\\ 0, &\forall x\in \dg_{[m]}, n\nmid m\end{array}\right.$  From (\ref{d+1}), (\ref{2.16}), (\ref{2.18}) we know that $f$ satisfies the conditions in Lemma \ref{lemm-8} and $\pi_{-1}(f)=\alpha_{-1}$. So we have proved $\pi_{-1}$ is surjective.
\end{proof}

 \begin{lemm}\label{lemm-10}The linear map  $\pi_0$  defined in (\ref{Eq2.13}) is surjective.\end{lemm}

 \begin{proof} Take $\alpha_{0}\in H_{(0)}$. For any $xt^a\in \dg_{(\mu)},yt^{b}\in \dg_{(\lambda)}
 	$  where $a, b\in\frac 1n\bZ, \lambda,\mu\in\bC$, we have
 	$$\alpha_{0}(xt^a,yt^b)=0\text{ if }a+b\ne0\text{ or }\lambda+\mu\ne 0.$$
 	So we take  $xt^a\in \dg_{(\mu)},yt^{-a}\in \dg_{(-\mu)}$. Formulas we will get   actually hold for any $x, y$ even they are not generalized eigenvectors with respect to $d$. Since
 	$(d+1)\dg_{(0)}=\dg_{(0)}$,
 	using Lemma \ref{lemm-6}(b) we know   that $\alpha_{0}(\fl_2  ,d\dg_{(0)}t^{-2})=0$. Furthermore
 	\begin{equation}\label{d}
 		\alpha_{0}(\fl_2  ,d\dg t^{-2})=0; \alpha_0(\fl_2, [x,f(d)y]t^{-2})=\alpha_0(\fl_2,[f(-d)x,y]t^{-2}),\forall f(t)\in \bC[t].\end{equation}
 	Applying this to  (\ref{o-m}), we have
 \begin{equation}\label{2.19}\begin{aligned}\alpha_{0}&\big((a+2k(d-1))xt^a,(a-k-kd)(a-kd)yt^{-a}\big)\\ &-\alpha_{0}\big((a+k(d-2))(a+k(d-1))xt^a,(a-2kd)yt^{-a}\big)\\=&-\binom k2\alpha_{0}\big(\fl_{2},[x,(a-2kd)(a-kd)y]t^{-2}\big)\\
 		&-\binom k2\alpha_{0}\big(\fl_{2},[(a+k(d-2))x,(a-2kd)y]t^{-2}\big)\\&
+\binom {2k}2\alpha_{0}(\fl_{2},[x,(a-k-kd)(a-kd)y]t^{-2})\\
=&-\binom k2\alpha_{0}(\fl_{2},[(a+2kd)(a+kd)x,y]t^{-2})\\
 		&-\binom k2\alpha_{0}(\fl_{2},[(a+2kd)(a+k(d-2))x,y]t^{-2})\\&
+\binom {2k}2\alpha_{0}(\fl_{2},[(a-k+kd)(a+kd)x,y]t^{-2})\\
=&k^2\alpha_{0}(\fl_{2},[a(d+a)x,y]t^{-2})+k^3\alpha_{0}(\fl_{2},[(d+a)(d-1)x,y]t^{-2}).\end{aligned}\end{equation}


From (\ref{lhd}), the coefficients of $k^3$ in (\ref{2.19}) give
$$\alpha((d-1)xt^a, (d-\frac 12)dyt^{-a})+\alpha((d-\frac 12)(d-1)xt^a,dyt^{-a})\\
=\frac 12 \alpha_{0}(\fl_2,[(d+a)(d-1)x,y]t^{-2}).
$$

Therefore, $$\alpha(xt^a, (d-\frac 12)dyt^{-a})+\alpha((d-\frac 12)xt^a,dyt^{-a})\\
=\frac 12\alpha_{0}(\fl_2,[(d+a)x,y]t^{-2}).
$$

 Let $B(x,y)=\alpha_{0}(xt^a,dyt^{-a})+\frac 12\alpha_{0}(\fl_2,[(d+a)x,y]t^{-2})$. Then
 $$\begin{aligned}B((d-\frac 12)x,y)&+B(x,(d-\frac 12)y)=\alpha_{0}((d-\frac 12)xt^a,dyt^{-a})+\frac 12\alpha_{0}(\fl_2,[(d+a)(d-\frac 12)x,y]t^{-2})\\ &+\alpha_{0}(xt^a,(d-\frac 12)dyt^{-a})+\frac 12\alpha_{0}(\fl_2,[(d+a)x,(d-\frac 12)y]t^{-2})=0.\end{aligned}$$ Again we have $B(x,y)=0$, that is

 \begin{equation}\label{02.23}\alpha_{0}(xt^a,dyt^{-a})=-\frac 12\alpha_{0}(\fl_2,[(d+a)x,y]t^{-2}).\end{equation}

 Exchanging $x$ and $y$, we deduce that $$\alpha_{0}(dxt^a,yt^{-a})=-\frac 12\alpha_{0}(\fl_2,[x,(d-a)y]t^{-2}).$$

Combining with (\ref{g-g-1}), we deduce that $$\begin{aligned}(a-k)\alpha_0&(xt^a,yt^{-a})-a\alpha_0(xt^{a-k},yt^{k-a})\\
	=&-k\alpha_0(dxt^a,yt^{-a})-k\alpha_0(xt^{a-k},dyt^{k-a})+\alpha_0(\fl_{k},[x,y]t^{-k})\\ =&(-k\frac a2+k\frac{a-k}2+\binom k2)\alpha(\fl_2,[x,y]t^{-2})=-\frac k2\alpha(\fl_2,[x,y]t^{-2}), \forall k\in \bZ,\end{aligned}$$ i.e.,

\begin{equation}\label{02.24}\aligned (a-k)(\alpha_0&(xt^a,yt^{-a})-\frac 12\alpha(\fl_2,[x,y]t^{-2}))\\
 =&a\Big(\alpha_0(xt^{a-k},yt^{k-a})-\frac 12\alpha(\fl_2,[x,y]t^{-2})\Big),\forall xt^a,yt^{-a}\in \fg, k\in \bZ.\endaligned\end{equation}

Now we can define the bilinear form on $B\in (\Bil(\ddot{\fg}))^{\sigma}$ by $$\begin{aligned} B(x,y):=&\frac 1a\Big(\alpha_0(xt^a,yt^{-a})-\frac 12\alpha(\fl_2,[x,y]t^{-2})\Big),\forall xt^a,yt^{-a}\in \fg,a\ne 0; \\ B(\partial,x)=&B(x,\partial)=\frac 12\alpha_0(\fl_2, xt^{-2}),\forall xt^{-2}\in \fg.\end {aligned}$$ It is straightforward to check that $B\in (\Inv(\ddot\fg))^{\sigma}$ and $\alpha_0=\pi_{0}(B)$. Hence $\pi_0$ is surjective.


\end{proof}

 \begin{lemm}\label{lemm-11} The linear map $\pi_1$  defined in (\ref{Eq2.14}) is surjective.\end{lemm}

 \begin{proof} Take $\alpha_{1}\in H_{(1)}$.  For any $xt^a\in \dg_{(\mu)},yt^{b}\in \dg_{(\lambda)}
 	$ where $a, b\in\frac 1n\bZ, \lambda,\mu\in\bC$, we have
 	$$\alpha_{1}(xt^a,yt^b)=0\text{ if }a+b\ne0\text{ or }\lambda+\mu\ne 1.$$
 	So we take  $xt^a\in \dg_{(\mu)},yt^{-a}\in \dg_{(1-\mu)}$. Formulas we will get actually hold for any $x, y$ even they are not generalized eigenvectors with respect to $d$.  From Lemma \ref{lemm-6}(b) we know   that 
 	$$\alpha_{1}(\fl_2  ,\dg_{(0)}t^{-2})=\alpha_{1}(\fl_i  ,\fl_j)=0, \forall i,j\in \bZ.$$
 	Using this to  (\ref{o-m}), we have
 \begin{equation}\label{2.23}\aligned \alpha_{1}\big((a+&2k(d-1))xt^a,(a-k-kd)(a-kd)yt^{-a}\big)\\ &=\alpha_{1}\big((a+k(d-2))(a+k(d-1))xt^a,(a-2kd)yt^{-a}\big).\endaligned\end{equation}

The coefficien of $k^3$ in (\ref{2.23}) gives $$\alpha_1((d-1)xt^a, (d-\frac 12)dyt^{-a})+\alpha_1((d-\frac 12)(d-1)xt^a,dyt^a))=0.$$ Recalling that $(d-1)\dg=\dg$, we may replace $(d-1)x$ with $x$ and replace $dy$ with $y$, to give
\begin{equation}\label{2.24}\alpha_{1}(xt^a,(d-\frac 12)yt^{-a})+\alpha_{1}((d-\frac 12)xt^a,yt^{-a})=0.\end{equation}

Switching  $x$ and $y$, and substituting $k$ with $-k$, $a$ with $-a$ in (\ref{g-g-1}),  we have $$\alpha_1(xt^a, (a+k(d-1))yt^{-a})-\alpha_1((a-kd)x t^{a-k},yt^{k-a})=0.$$ Combining with (\ref{g-g-1}) and (\ref{2.24}), we have
$$\begin{aligned}&0=\alpha_1((a+k(d-1))xt^a,yt^{-a})-\alpha_1(xt^{a-k},(a-kd)yt^{k-a})\\&+\alpha_1(xt^a, (a+k(d-1))yt^{-a})-\alpha_1((a-kd)x t^{a-k},yt^{k-a})\\ &=(2a-k)
\Big(\alpha_1(xt^a,yt^{-a})-\alpha_1(xt^{a-k},yt^{k-a})\Big),\end{aligned}$$
which implies \begin{equation}\label{2.25}\alpha_1(xt^a,yt^{-a})=\alpha_1(xt^{a-k},yt^{k-a}),\forall k\in \bZ.\end{equation}

Now from (\ref{2.25}) and (\ref{2.24}), we have $\dot{\alpha}\in (Z^2(\dg))^{d,\sigma}$ defined by $$\dot{\alpha}(x,y)=\alpha_1(xt^a,yt^{b}),\forall x\in \dg_{[na]},y\in \dg_{[nb]}.$$ And it is easy to see that $\alpha_1=\pi_1(\dot\alpha)$, hence $\pi_1$ is surjective.
 \end{proof}

{\it Proof of Theorem \ref{thm1}}. From Lemma \ref{lemm-5},\ref{lemm-7}-\ref{lemm-11}, we have $$H^2(\fL(\dg,d,\sigma),\bC^{1|1})_{\bar 0}\cong \big(Z^2(\dg)\big)^{d,\sigma}\oplus {(\Inv(\ddot{\fg}))}^\sigma\oplus \big(\dg/((d+1)\dg+[(d+\frac 12)\dg,\dg]+ [\dg,[\dg,\dg]])\big)^{\sigma}.$$ So we only need to show that $\big(Z^2(\dg)\big)^{d,\sigma}\cong H^2(\dg,\bC^{1|1})_{\bar 0}^{d,\sigma}$.
 In fact, for any $\alpha\in \big(B^2(\dg)\big)^{d,\sigma}$, there exists a linear map $f:L\rightarrow \bC$, such that $$\alpha(x,y)=f([x,y]), \text{ and } \alpha(x,y)=\alpha(dx,y)+\alpha(x,dy),\forall x,y\in \dg.$$ Thus $f([x,y])=f([dx,y])+f([x,dy])$,  i.e., $f((d-1)[x,y])=0$. Since $(d-1)[\dg,\dg]=[\dg,\dg]$, we have $f([\dg,\dg])=0$, i.e., $\alpha=0$. So $\big(B^2(\dg)\big)^{d,\sigma}=0$, which implies $\big(Z^2(\dg)\big)^{d,\sigma}\cong H^2(\dg,\bC^{1|1})_{\bar 0}^{d,\sigma}$ as desired.\qed

\

Let $\rho_{i,\bar{l}},i=1,\ldots,n_{-1,\bar{l}}$ be a basis of $\big((\dg/\big((d+1)\dg+[(d+\frac 12)\dg,\dg]+[\dg,[\dg,\dg]]\big))_{\bar l}^{\sigma}\big)^*$;
$B_{j,\bar{l}},j=1,\ldots,n_{0,\bar l}$ be a basis of $\{B\in {\Inv(\ddot{\fg})}_{\bar l}^\sigma|B(\partial,\partial)=0\}$;  $\dot{\alpha}_{k,\bar{l}},k=1,\ldots,n_{1,\bar{l}}$ be a basis $Z^2(\dg)_{\bar l}^{d,\sigma}$, $Z=Z_{\bar 0}\oplus Z_{\bar 1}$, and $Z$ has a basis $$\{z, z_{-1,i,\bar{l}},z_{0,j,\bar{l}},z_{1,k,\bar{l}}|i=1,2,\ldots, n_{-1,\bar{l}}; j=1,2,\ldots,n_{0,\bar{l}};k=1,2,\ldots, n_{1,\bar{l}},l=0,1\}.$$

Then we have the universal central extensions $\hL(\dg,d,\sigma)=\fL\oplus Z$ of $\fL$ with brackets:
 \begin{equation}\label{hL}\begin{aligned}{ } [\fl_k,\fl_j]=&(j-k)\fl_{k+j}+\delta_{k+j,0}\frac{k^3-k}{12}z,\\
 [\fl_k, xt^a]=& (a+kd)xt^{a+k}+\sum_{i,\bar{l}} \delta_{k+a,0}\frac{k^3-k}{12} \rho_{i,\bar{l}}(x)z_{-1,i,\bar{l}}\\
 &+\sum_{i,\bar{l}}\delta_{k+a,0}(k^2-k) B_{i,\bar{l}}(\partial,x)z_{0,i,\bar{l}},\\
  [xt^a,yt^b]=&[x,y]t^{a+b}+\sum_{i,\bar{l}}\delta_{a+b,0}\frac{1-4a^2}{24}\rho_{i,\bar{l}}([x,y] )z_{-1,i,\bar{l}}\\
  &+\sum_{i,\bar{l}}\delta_{a+b,0} \Big(a B_{i,\bar{l}}(x,y)+B_{i,\bar{l}}(\partial,[x,y])\Big)z_{0,i,\bar{l}}\\
  &+\sum_{i,\bar{l}}\delta_{a+b,0}\dot\alpha_{i,\bar{l}}(x,y)z_{1,i,\bar{l}},\\
  [\fL,Z]=&0.\end{aligned}\end{equation}



\begin{section}{$A\fL$-modules }\end{section}

We  need first to recall the   algebras: $A, W, \dg, \mathfrak{g}$ defined  in Section 1.
Now define  $\tL=\fL\ltimes A$ where $[A,A]=0, [\mathfrak{g}, A]=0, [\l_i, t^j]=jt_{i+j}$ for $i,j\in\bZ$. A $\tL$-module $P$ is called an $A\fL$-module  if $A$ acts associatively on $P$, i.e., $t^it^j v=t^{i+j}v,t^0 v=v$ for all $i,j\in \bZ, v\in P$.

In this section, we will determine all simple   cuspidal $A\fL$-modules which will be used to determine all simple quasi-finite modules over the Lie algebras $\hL$ in Section 4. We will first set up eight auxiliary results.


We will apply the weighting functor $\mathfrak{W}$ introduced in \cite{N2}. For any $A\fL$-module $P$ and  $\lambda\in\bC$, denote $$\mathfrak{W}^{(\lambda)}(P):=\bigoplus_{a\in\frac 1n \bZ}\Big(\big( P/ (\fl_0-\lambda-a)P\big)\otimes t^{a}\Big).$$
By Proposition 8 in \cite{N2}, we know that $\mathfrak{W}^{(\lambda)}(P)$ is an  $A\fL$-module with the actions

$$x\cdot\big((v+(\fl_0-\lambda-a)P)\otimes t^{a}\big):= \big(xv+(\fl_0-\lambda-a-r)P\big)\otimes t^{a+r},\forall x\in \tL_r,v\in P,a\in \frac 1n \bZ.$$

It is clear that $\mathfrak{W}^{(\lambda)}(P)$ is a weight   $A\fL$-module. If  $P$ is a weight module with $\Supp(P)\subseteq \lambda+\frac 1n \bZ$, then $\mathfrak{W}^{(\lambda)}(P)=P$. If  $P$ is a weight module with $\Supp(P)\cap (\lambda+\frac 1n \bZ)=\emptyset$, then $\mathfrak{W}^{(\lambda)}(P)=0$.

Now for any $\ca:=(t-1)W\ltimes \fg$ module $V$, we make it into an $\ca\ltimes A$ module by $t^i v=v,$ for all $i\in \bZ, v\in V$. Note that elements in $(t-1)W$ are linear combinations of  elements of the form $\fl_i-\fl_j$. Then we have the induced $A\fL$-module
$$\widetilde{V}:=\Ind_{(t-1)W\ltimes (\fg\oplus A)}^{\tL} V=\bC[\fl_0]\otimes V.$$

Note that $\tilde{V}$ is $\bC[\fl_0]$ free.  By identifying the vector space $V$ with  the vector spaces $\widetilde{V}/(\fl_0-\lambda-a)\widetilde{V}$ for all $a\in \frac 1n \bZ$, we have $$\mathfrak{W}^{(\lambda)}(\widetilde{V})=V\otimes \bC[t^{\frac 1n},t^{-\frac 1n}]$$ with the actions
\begin{eqnarray*}x (v\otimes t^{a}) &=& xv\otimes t^{a+b},\forall x\in \fg_{b}, \\
t^j (v\otimes t^{a}) &=& v\otimes t^{a+j},\\
\fl_j (v\otimes t^{a}) &=&(\lambda+a+j-\fl_0+\fl_j)v\otimes t^{a+j},\forall j\in \bZ,v\in V,a,b\in \frac 1n \bZ,
\end{eqnarray*}
where in the last equation, $v\in \widetilde{V}/(\fl_0-\lambda-a)\widetilde{V}$ on the left hand side, at the same time  $v\in \widetilde{V}/(\fl_0-\lambda-a-j)\widetilde{V}$ on the right hand side.

Let $f:L_1\rightarrow L_2$ be any homomorphism of Lie superalgebras and $V$ be a $L_2$ module, then we have the $L_1$ module $V^{f}=V$  with action $x\circ v=f(x) v,\forall x\in L_1,v\in V$. Let $\tau$ be the automorphism of $\tL$ with $\tau(\fl_j)=\fl_j-jt^j,\tau(x)=x,\forall x\in \fg\oplus A, j\in \bZ$. Then we have the $A\fL$-module

$$\Gamma(V, \lambda):=(\mathfrak{W}^{(\lambda)}(\widetilde{V}))^\tau=V\otimes \bC[t^{\frac 1n},t^{-\frac 1n}]$$ with the actions
\begin{eqnarray}\label{weighted-1}x (v\otimes t^{a}) &=& xv\otimes t^{a+b},\forall x\in \fg_b \\
\label{weighted-2}  t^j (v\otimes t^{a}) &=& v\otimes t^{a+j},\\
\label{weighted-3}  \fl_j (v\otimes t^{a}) &=&(\lambda+a-\fl_0+\fl_j)v\otimes t^{a+j},\forall j\in \bZ,v\in V,a,b\in \frac 1n \bZ.
\end{eqnarray}

Note that
 $\ca$ has a $\bZ_n$-gradation $\ca=\oplus_{i\in \bZ_n} \ca_{[i]}$ with $$\ca_{[i]}=\dg_{[i]}\otimes t^{\frac in}\bC[t,t^{-1}]\oplus \delta_{[i],[0]} (t-1)W,\forall i\in \bZ.$$ 
 Furthermore, suppose that $V$ is a $\bZ_n$-graded $\ca$-module, i.e., $V$ has a supersubspace decomposition $V=\oplus_{i=0}^{n-1} V_{[i]}$ with $\ca_{[i]}\cdot V_{[j]}\subseteq V_{[i+j]},\forall [i],[j]\in \bZ_n$. Then \begin{equation}\Gamma(V,\lambda)=\oplus_{i=0}^{n-1} M_{i}.\end{equation} 
 with $M_i=\oplus_{j\in \bZ} \big(V_{[j]}\otimes t^{\frac {j+i}n}\big)$ are $A\fL$-submodules of $\Gamma(V,\lambda)$. Denote \begin{equation}F(V,\lambda):=\oplus_{j\in \bZ} V_{[j]}\otimes t^{\frac {j}n}\subseteq \Gamma(V,\lambda).\end{equation}

 We call $\Gamma(V,\lambda)$ and $F(V,\lambda)$ as tensor modules or loop modules over $\fL$ (resp. over $\hat{\fL}$ with $Z$ acting as zero).

Let $M$ be a weight $A\fL$-module with  $\supp(M)\subseteq \lambda+\frac 1n \bZ$ for some $\lambda\in \bC$. Then $(t-1)M$ is an $\ca$-module, and $M/(t-1)M$ is a $\bZ_n$-graded $\ca$-module with $$(M/(t-1)M)_{[i]}=M_{\lambda+\frac in}+(t-1)M,\forall i\in \bZ.$$

\begin{prop}\label{prop-13}Let $M$ be a weight $A\fL$-module with  $\supp(M)\subseteq \lambda+\frac 1n \bZ$ for some $\lambda\in \bC$. Then $M\cong F((M/(t-1)M,\lambda)$. \end{prop}

\begin{proof}It is easy to see that the following linear map  is bijective:
	$$\aligned \psi: &M\rightarrow F(M/(t-1)M,\lambda),\\
	&\psi(v_{\lambda+a})=\overline{v_{\lambda+a}}\otimes t^{a},\forall a\in \frac 1n\bZ,v_{\lambda+a}\in M_{\lambda+a},\endaligned$$ where $\bar{v}=v+(t-1)M$ for all $v\in M$. We are going to show that $\psi$ is an isomorphism of $A\fL$-modules. In fact, from (\ref{weighted-3}), we have
	$$\begin{aligned}\fl_j \psi(v_{\lambda+a})=&\fl_j (\overline{v_{\lambda+a}}\otimes t^a)\\
		=&((\lambda+a-\fl_0+\fl_j)\overline{v_{\lambda+a}})\otimes t^{a+j}\\
		=&\overline{(\lambda+a-\fl_0+\fl_j)v_{\lambda+a}}\otimes t^{a+j}\\
		=&\overline{\fl_jv_{\lambda+a}}\otimes t^{a+j}\\
		=&\psi(\fl_jv_{\lambda+a}).\end{aligned}$$
And $\psi(x v)=x\psi(v)$ for all $x\in \fg,v\in M$ follows directly from (\ref{weighted-1}) and (\ref{weighted-2}). \end{proof}

\begin{prop}\label{prop12} Suppose that $V$ is a finite dimensional $\bZ_n$-graded $\ca$-module.  

 (1) The loop module $F(V,\lambda)$ is a simple $A\fL$-module if and only if $V$ is $\bZ_n$-graded-simple, i.e., $V$ has no nontrivial $\bZ_n$-graded $\ca$-submodule.

 (2) The loop module  $\Gamma(V,\lambda)$ is completely reducible if $V$ is $\bZ_n$-graded-simple.\end{prop}

\begin{proof} (1). If $V$ is not $\bZ_n$-graded-simple, then it has a nontrivial $\bZ_n$-graded-simple submodule $V'$. By definition, $F(V',\lambda)$ is a nontrivial $A\fL$ submodule of $F(V,\lambda)$, hence $F(V,\lambda)$ is not simple.  
	
	Now suppose that $M=F(V,\lambda)$ is not simple, then it has a nontrivial $A\fL$ submodule $M'$.  So  $M/(t-1)M$ has a nontrivial $\bZ_n$-graded $\ca$ submodule $M'/(t-1)M'$.  And from (\ref{weighted-1}-\ref{weighted-3}), we have the nature $\bZ_n$-graded $\ca$ module isomorphism $$V\rightarrow M/(t-1)M, v_{[i]}\mapsto v\otimes t^{\frac in}+(t-1)M,\forall v_{[i]}\in V_{[i]}.$$ So $V$ is not $\bZ_n$-graded-simple, and we have (1).

(2). Note that $K_i:=\oplus_{j\in \bZ} V_{[j-i]}\otimes t^{\frac jn}$ for $i=0,1,\ldots,n-1$ are $A\fL$-submodule of $\Gamma(V,\lambda)$. We have $\Gamma(V,\lambda)=\oplus_{i=0}^{n-1} K_i$, and $K_i=F(V,\lambda+\frac in)$ are simple $A\fL$-modules if $V$ is $\bZ_n$-graded-simple. Statement (2) holds.
\end{proof}

Note that all finite dimensional simple $\bZ_n$-graded $\ca$-modules for Lie algebra $\ca$ were classified in \cite{MZ}.

We also need the following two lemmas, which is similar to Lemma 2.4 and 2.5 in \cite{CLW}.

\begin{lemm}\label{rel-subalg}
Let $k,l\in\bZ_+,i,j\in \bZ, xt^a\in \fg$. Then we have
\begin{enumerate}
\item $[(t-1)^k\fl_i,(t-1)^l \fl_j]=(l-k+j-i)(t-1)^{k+l}\fl_{i+j}+(l-k)(t-1)^{k+l-1}\fl_{i+j}$;
\item $[(t-1)^k\fl_i,x(t-1)^lt^a]=(a+id)x(t-1)^{k+l}t^{i+a}+(l+kd)x(t-1)^{k+l-1}t^{i+a+1}$.
\end{enumerate}
\end{lemm}

\begin{lemm}\label{ideal}
For $k\in\bZ_+$, let $\fa_k=(t-1)^{k+1}W\ltimes((t-1)^k \fg)$. Then
\begin{enumerate}
\item $\fa_k$ is an ideal of $\fa_0=\fa$ and $\fa/\fa_1\cong\ddg$;
\item $[\fa_1,\fa_k]\subseteq\fa_{k+1}$;
\item the ideal of $\fa$ generated by $(t-1)^kW$ contains $\fa_{k}$;
\item $[\fa_{\bar 0},\fa_{\bar 0}]\supseteq\fa_{1,\bar{0}}$.
\end{enumerate}
\end{lemm}

\begin{lemm}[{\cite[Proposition 19.1]{H}}]\label{reductive}
\begin{enumerate}
\item Let $L$ be a finite dimensional reductive Lie algebra. Then $L=[L,L]\oplus Z(L)$ and $[L,L]$ is semisimple.
\item Let $L\subseteq\gl(V)$ ($\dim V<\infty$) be a Lie algebra acting irreducibly on $V$. Then $L$ is reductive and $\dim Z(L)\leq1$.
\end{enumerate}
\end{lemm}

\begin{lemm}[{\cite[Theorem 2.1]{M}}, Engel's Theorem for Lie superalgebras]\label{Engel}
Let $V$ be a finite dimensional module for the Lie superalgebra $L=L_{\bar{0}}\oplus L_{\bar{1}}$ such that the elements of $L_{\bar{0}}$ and $L_{\bar{1}}$ respectively are nilpotent endomorphisms of $V$. Then there exists a nonzero element $v\in V$ such that $xv=0$ for all $x\in L$.
\end{lemm}

\begin{lemm} \label{graded-simple} Let $G$ be an additive group, $L$ be a finite dimensional $G$-graded Lie superalgebra, $\frak{n}$ be a $G$-graded nilpotent ideal of $L$ with $\frak{n}_{\bar 0}\subseteq [L_{\bar 0},L_{\bar 0}]$. Then for any finite dimensional $G$-graded-simple $L$ module $V$, we have $\frak{n} V=0$. \end{lemm}

\begin{proof} Let $M$ be any finite-dimensional simple $L_{\bar 0}$ module. From Lemma \ref{reductive} we know that  $L_{\bar 0}/\ann_{L_{\bar 0}}(M)$ is reductive,  where $\ann_{L_{\bar 0}}(M)=\{x\in L_{\bar 0}|xM=0\}$. Moreover, $[L_{\bar 0}/\ann_{L_{\bar 0}}(M),L_{\bar 0}/\ann_{L_{\bar 0}}(M)]$ is a semisimple Lie algebra. And from $\frak{n}_{\bar 0}\subseteq [L_{\bar 0},L_{\bar 0}]$ we know that $(\frak{n}_{\bar 0}+\ann_{L_{\bar 0}}(M))/\ann_{L_{\bar 0}}(M)$ is a nilpotent ideal of the semisimple Lie algebra $[L_{\bar 0}/\ann_{L_{\bar 0}}(M),L_{\bar 0}/\ann_{L_{\bar 0}}(M)]$, which implies that $\frak{n}_{\bar 0}\subseteq \ann_{L_{\bar 0}} M$. 
	
	Applying the above established result to a composition series of $L_{\bar 0}$ submodules of the $L_{\bar 0}$ module $V$:
	$$ V\supset V_1\supset V_2\supset\cdots \supset V_r=\{0\},$$
	we see that any element in $\frak{n}_{\bar 0}$ acts nilpotently on $V$. And from $[x,x]\in \frak{n}_{\bar 0},\forall x\in \frak{n}_{\bar 1}$, we know that any element in $\frak{n}_{\bar 1}$ acts nilpotently on $V$. Let $V'=\{v\in V| \frak{n} v=0\}$. From Engel's Theorem for Lie superalgebras we know that $V'\ne 0$. It is easy to verify that $V'$ is a $G$-graded $L$ submodule of $V$. So $V'=V$, i.e., $\frak{n} V=0$.  \end{proof}

\begin{prop} \label{prop-17} For any finite dimensional simple (resp. $\bZ_n$-graded-simple) $(t-1)W\ltimes \fg$ module $V$, we have $\fa_1\cdot V=0$. Hence $V$ is a simple (resp. $\bZ_n$-graded-simple) module over  $\fa/\fa_1\cong\ddg$.\end{prop}

\begin{proof}Note that $V$ is a finite dimensional $(t-1)W$ module. From Lemma 2.6 in \cite{CLW}, we have $(t-1)^{k} W\cdot V=0$ for some $k\in \bN$.

 From Lemma \ref{ideal} (3), we know that  $\fa_{k}\cdot V=0$.  Hence $V$ is a simple(resp. $\bZ_n$-graded-simple) module over $\fa/\fa_k$. From Lemma \ref{ideal}, we may apply Lemma \ref{graded-simple} for $L=\fa/\fa_k$ and $\frak{n}=\fa_1/\fa_k$ to obtain $\fa_1 V=0$ as expected.  \end{proof}

Now for any  $\ddg$ module $V$, using Proposition \ref{prop-17} we can naturally regard it into  a $(t-1)W\ltimes (\fg\oplus A)$ module by $ \fa_{1}\cdot V=0$ and $t^i(v)=v,$ for all $v\in V$. Then we have the tensor modules $\Gamma(V,\lambda)$ and $F(V,\lambda)$. More precisely,
for any  $\ddg$-module $V$ and $\lambda\in \bC$, we have $A\fL$ weight module $\Gamma(V,\lambda):=V\otimes \bC[t^{\frac 1n},t^{-\frac 1n}]$ with actions
\begin{align}
&t^i\cdot(v\otimes t^b)=v\otimes t^{b+i},\\
&\fl_i\cdot(v\otimes t^b)=(\lambda+b+i\partial )v\otimes t^{b+i},\\
&xt^a \cdot(v\otimes t^b)=xv\otimes t^{a+b}, \forall xt^a\in \fg, i\in \bZ, b\in \frac 1n \bZ, v\in V.
\end{align}

And for any $\bZ_n$-graded $\ddg$-module $V$, i.e., $V=\oplus_{i=0}^{n-1} V_{[i]}$ with $\ddg_{[i]}\cdot V_{[j]}\subseteq V_{[i+j]}$ for all $[i],[j]\in \bZ_n$, where $\ddg_{[i]}=\dg_{[i]}\oplus \delta_{[i],[0]} \bC \partial,\forall i\in \bZ$. We have the $A\fL$-module
\begin{equation}F(V,\lambda):=\oplus_{j\in \bZ} V_{[j]}\otimes t^{\frac jn}\subseteq \Gamma(V,\lambda).\end{equation}

Now we are ready to give the classification of simple cuspidal $A\fL$-modules.

\begin{theo}\label{AL}Let $M$ be any simple cuspidal $A\fL$-module. Then

(1) $M\cong F(V,\lambda)$ for some $\lambda\in \bC$ and some finite dimensional $\bZ_n$-graded-simple $\ddg$ module $V$;

(2) $M$ is isomorphic to a simple $\tilde{\fL}$ sub-quotient of a loop module $\Gamma(V',\lambda)$ for some $\lambda\in \bC$ and some finite dimensional simple $\ddg$ module $V'$.
\end{theo}

\begin{proof} Statement (1) follows from Propositions \ref{prop-13}, \ref{prop12}, \ref{prop-17} and the fact that $M/(t-1)M$ is finite dimensional.
	
	(2). From Proposition \ref{prop-13}, $M$ is isomorphic to a submodule of $\Gamma(M/(t-1)M,\lambda)$. Since $M/(t-1)M$ is finite dimensional, it  has a composition series of $\fa$-modules  $0=V_0\subset V_1\subset\cdots \subset V_k=M/(t-1)M$ with $V_i/V_{i-1}$ are simple $\fa$-modules. Then $\Gamma(M/(t-1)M,\lambda)$ has a filtration  $$0\subset \Gamma(V_1,\lambda)\subset\cdots\subset \Gamma(V_i,\lambda)\cdots \subset \Gamma(M/(t-1)M,\lambda).$$ Since $\Gamma(M/(t-1)M,\lambda)$ has finite length, its simple $A\fL$ submodule $F(M/(t-1)M,\lambda)$ hence $M$ is isomorphic to a simple sub-quotient  of $\Gamma(V_i/V_{i-1},\lambda)\cong \Gamma(V_i,\lambda)/\Gamma(V_{i-1},\lambda)$. So we have proved (2). \end{proof}

\begin{remark}The method used in this section turns out to be very general. 
 And its application to superconformal algebras is in process.\end{remark}

\begin{section}{Classification of quasi-finite modules}\end{section}

In this section, we will determine all simple    quasi-finite modules over the Lie algebras $\hL$ (certainly including $\fL$), i.e., to prove Theorem \ref{thm2}.

Recall that in \cite{BF}, the authors show that every cuspidal $\Vir$-module is annihilated by the operators $\Omega_{k,s}^{(m)}$ for enough large $m$.

\begin{lemm}[{\cite[Corollary 3.7]{BF}}]\label{Omegaoper}
For every $\ell\in\bN$ there exists $m\in\bN$ such that for all $k, s\in\bZ$ the differentiators $\Omega_{k, s}^{(m)}=\sum\limits_{i=0}^m(-1)^i\binom{m}{i}\fl_{k-i}\fl_{s+i}$ annihilate every cuspidal $\Vir$-module with a composition series of length $\ell$.
\end{lemm}

In the next four lemmas we will show that any simple cuspidal $\fL$-module is a simple quotient of a simple cuspidal $A\fL$-module.

Let $M$ be a cuspidal $\fL$-module with $\dim M_\lambda\leq N,$ for all $\lambda\in\supp(M)$. Then for any $\lambda\in\supp(M)$, $\bigoplus\limits_{i\in\bZ}M_{\lambda+i}$ is a cuspidal $\Vir$-module (the center acts trivially) with length $\leq 2N$. 
Hence there exists $m\in\bN$ such that $$\Omega_{k,p}^{(m)}\in \ann_{U(\fL)}(M), \forall k,p\in\bZ.$$ Therefore,
$$\begin{aligned}f_0(a,k,p):=&[\Omega_{k,p}^{(m)},x(a)]\\
	=
\sum\limits_{i=0}^{m}&(-1)^i\binom{m}{i}\Big(\big(a+(k-i)d\big)x(a+k-i)\fl_{p+i}+\fl_{k-i}\big(a+(p+i)d\big)x(a+p+i)\Big)\\ &\in \ann_{U(\fL)}(M).\end{aligned}$$

We compute $$\begin{aligned} f_1(a,k,p):=&f_0(a+1,k,p-1))-f_0(a,k,p)\\
	=&\sum\limits_{i=0}^m(-1)^i\binom{m}{i}\Big((a+1+(k-i)d)x(a+1+k-i)\fl_{p-1+i}\\
	&-(a+(k-i)d)x(a+k-i)\fl_{p+i}+\fl_{k-i}(1-d)x(a+p+i)\Big)\in \ann_{U(\fL)}(M).\end{aligned}$$
Thus $$\begin{aligned} f_2(a,&k,p):=f_1(a,k,p)-f_1(a-1,k,p+1)\\
	=&\sum\limits_{i=0}^{m}(-1)^i\binom{m}{i}\Big((a+1+(k-i)d)x(a+1+k-i)\fl_{p-1+i}\\
	&-2(a+(k-i)d)x(a+k-i)\fl_{p+i}+
(a-1+(k-i)d)x(a-1+k-i)\fl_{p+1+i}\Big)\\
&\in \ann_{U(\fL)}(M).\end{aligned}$$
Then \begin{align*}
f_2(&a,k,p)-f_2(a-1,k+1,p)\\
=&\sum\limits_{i=0}^m(-1)^i\binom{m}{i}((1-d)x(a+1+k-i)\fl_{p-1+i}-2(1-d)x(a+k-i)\fl_{p+i}\\ &+(1-d)x(a-1+k-i)\fl_{p+1+i})\\
=&\sum\limits_{i=0}^{m+2}(-1)^i\binom{m+2}{i}(1-d)x(a+k+1-i)\fl_{p-1+i}\in \ann_{U(\fL)}(M),
\end{align*}

i.e.,

\begin{align}
[\Omega_{k,p-1}^{(m)}&,x(a+1)]-2[\Omega_{k,p}^{(m)},x(a)]+[\Omega_{k,p+1}^{(m)},x(a-1)]\nonumber\\
&-[\Omega_{k+1,p-1}^{(m)},x(a)]+2[\Omega_{k+1,p}^{(m)},x(a-1)]-[\Omega_{k+1,p+1}^{(m)},x(a-2)]\nonumber\\
=&\sum\limits_{i=0}^{m+2}(-1)^i\binom{m+2}{i}(1-d)x(a+k+1-i)\fl_{p-1+i}\in \ann_{U(\fL)}((M)\label{Omega}.
\end{align}Recall that we have assumed that $1$ is not an eigenvalue of $d$. Therefore we have established the following result.

\begin{lemm}\label{omega}
Let $M$ be a cuspidal module over $\fL$. Then there exists $m\in\bN$ such that for all $p\in\bZ, x(a)\in \fg, x\in  \dg$, the operator 
$\Omega_{x(a),p}^{(m)}=\sum\limits_{i=0}^m(-1)^i\binom{m}{i}x(a-i)\fl_{p+i}\in U(\fL)$ annihilate $M$.\end{lemm}

Now let $M$ be a cuspidal simple $\fL$-module. Then $\fg M$ is a $\fL$ submodule, which has to be zero or $M$. If $\fg M=0$, then $M$ is a simple cuspidal $W$ module which is clearly described in \cite{Ma}.
Now we assume that $\fg M=M$.
Consider $\fg$ as the adjoint $\fL$-module. Then we have the tensor product $\fL$-module  $\fg\otimes M$, which  becomes an $A\fL$-module under
\[
x\cdot (y\otimes u)=(xy)\otimes u, \forall x\in A,y\in\fg, u\in M.
\]
This is not hard to verify.

For any $x\otimes t^b\in \fg$ with $b\in \frac 1n \bZ$ and $k\in\bZ$, by
$t^k(xt^b)$
 we mean $xt^{k+b}$.
Denote $$\fK(M)=\{\sum\limits_{i}x_i\otimes v_i\in \fg\otimes M\,|\,\sum\limits_{i}(t^kx_i)v_i=0, \forall k\in \bZ\}.$$ It is straightforward but tedious to verify the following result.

\begin{lemm}\label{submod} The subspace $\fK(M)$ is an $A\fL$ submodule of $\fg\otimes M$. 
\end{lemm}

 Hence we have the $A\fL$-module $\widehat{M}=(\fg\otimes M)/\fK(M)$. Also, we have a $\fL$-module epimorphism defined by
\[\pi: \widehat{M}\to \fg M; x\otimes y+\fK(M)\mapsto xy, \forall x\in \fg, y\in M.\]
$\widehat{M}$ is called the $A$-cover of $M$.

\begin{lemm}\label{covercuspidal}
For any  simple cuspidal $\fL$-module $M$ with $\fg M\ne 0$, the $A\fL$-module $\widehat{M}$ is cuspidal.
\end{lemm}
\begin{proof} Suppose that $\Supp(M)\subseteq \lambda+\frac{1}{n}\bZ$ and $\dim M_\mu\leq r$ for all $\mu\in\Supp(M)$. From Lemma \ref{omega}, there exists $m\in\bN$ such that for all $p\in\bZ, x(a)\in \fg, x\in  \dg$, the operators
$\Omega_{x(a),p}^{(m)}=\sum\limits_{i=0}^m(-1)^i\binom{m}{i}x(a-i)\fl_{p+i}$ annihilate $M$. Hence
\begin{equation}\label{4.2}\sum\limits_{i=0}^m(-1)^i\binom{m}{i}x(a-i)\otimes \fl_{p+i}v\in \fK(M),\forall v\in M,x(a)\in A.\end{equation}

Let $S=\sum_{j=0}^{n-1}\sum_{i=0}^{m-1} (\dg_{[j]}\otimes t^{\frac jn -i})\otimes M+\fg\otimes M_0$. It is not hard to show that $S$ is a $\bC \fl_0$ submodule of $\fg\otimes M$ with \[\dim S_\mu\leq 2mr\dim \dg , \forall \mu\in\lambda+\frac{1}{n}\bZ.\]
We will prove that $\fg\otimes M=S+\fK(M)$, from which we know $\widehat{M}$ is cuspidal. Indeed, we will prove by induction on $i$ that for all $x(\frac jn)\in \fg$, $u\in M_\mu$ with $\mu \ne 0$, $j=0,1,\ldots,n-1$,
\[
x(\frac jn+l)\otimes u_{\mu}\in S+\fK(M),\forall l\in \bZ.
\]
We only prove the claim for $l>m$, the proof for $l<0$ is similar. Then by \eqref{4.2} and induction hypothesis, we have
\begin{align*}x(\frac jn+l)\otimes u &=\frac 1\mu x(\frac jn+l)\otimes \fl_0u\\ &=\frac 1\mu \Big(\sum\limits_{i=0}^m(-1)^i\binom{m}{i}x(\frac jn+l-i)\otimes \fl_i u-\sum\limits_{i=1}^{m}(-1)^i\binom{m}{i}x(\frac jn+l-i)\otimes \fl_i u\Big)\\ &\in S+\fK(M).\end{align*}
\end{proof}

Now we can classify all simple cuspidal $\fL$-modules.
\begin{theo}\label{cuspidal} Any simple cuspidal $\fL$-module is a simple quotient of a simple cuspidal $A\fL$-module.

\end{theo}
\begin{proof} It is obvious if $\fg M=0$. So we may assume that $\fg M=M$ and there is an epimorphism $\pi: \widehat{M}\to M$. From Lemma \ref{covercuspidal}, $\widehat{M}$ is cuspidal. Hence $\widehat{M}$ has a composition series of $A\fL$ submodules:
\[
0=\widehat{M}^{(0)}\subset\widehat{M}^{(1)}\subset\cdots\subset\widehat{M}^{(s)}=\widehat{M}
\]
with $\widehat{M}^{(i)}/\widehat{M}^{(i-1)}$ being simple $A\fL$-modules. Let $k$ be the minimal integer such that $\pi(\widehat{M}^{(k)})\neq0$. Then we have $\pi(\widehat{M}^{(k)})=M, \widehat{M}^{(k-1)}=0$ since $M$ is simple. So we have an $\fL$-epimorphism from the simple $A\fL$-module $\widehat{M}^{(k)}/\widehat{M}^{(k-1)}$ to $M$.
\end{proof}



The following result for Virasoro algebra is well-known.

\begin{lemm}\label{weightupper}
Let $M$ be a   quasi-finite weight module over the Virasoro algebra with $\mathrm{supp}(M)\subseteq\lambda+\bZ$. If for any $v\in M$, there exists $N(v)\in\bN$ such that $\fl_iv=0$ for all $i\geq N(v)$, then $\mathrm{supp}(M)$ is upper bounded.
\end{lemm}

Now we are ready to determine all simple quasi-finite modules  over $(d,\sigma)$-twisted Affine-Virasoro superalgebra $\hL=\hL(\dg,d,\sigma)$ defined in (\ref{hL}).

\begin{theo}\label{noncuspidal}Any simple  quasi-finite weight module over $\hL$ is a highest weight module, a lowest weight module, or a cuspidal weight module. \end{theo}

\begin{proof}Let $M$ be a simple quasi-finite $\hat{\fL}$  with $\lambda\in \supp(M)$. Then $\supp(M)\subseteq \lambda +\frac 1n \bZ$. Suppose that $M$ is  not cuspidal. By retaking $\lambda$ we can find   some $a=-k+\frac jn\in \frac 1n\bZ$ with $k\in \bZ$ and $0\le j<n$, such that $$\dim M_{\lambda+a}\ge \dim \Hom_{\bC}(\oplus_{i=0}^{2n} \hL_{k+\frac in},\oplus_{i=0}^{3n}  M_{\lambda+\frac in}).$$ Without loss of generality, we may assume that $k>0$. Then the linear map $$\kappa:  M_{\lambda+a}\rightarrow \Hom_{\bC}(\oplus_{i=0}^{2n} \hL_{k+\frac in},\oplus_{i=0}^{3n}  M_{\lambda+\frac in})$$ defined by $ \kappa(v)(x)=xv$ has nonzero kernel. So there exists a nonzero homogeneous $v\in M_{\lambda+a}$ such that $(\sum_{i=0}^{2n} \hL_{k+\frac in}) v=0$. We can find $a_0\in \bZ_+$, such that $\oplus_{ a\ge a_0}\fL_{a}$ is contained in the subalgebra generated by $\fl_{k},\fl_{k+1}$ and $\fg_{k+\frac in},i=0,1,\ldots, n-1$. So $\oplus_{ a\ge a_0}\fL_{a}\subseteq \ann_{U(\hL)}(v)$. By exchange $\fl_{k}$ from left  to right, we have $$\fl_{k}\fL_{a_1}\cdots \fL_{a_l}\subseteq \sum_{i\ge k-|a_1|-\cdots-|a_l|} U(\hL)\fL_{i}$$ for sufficient large $k$, which implies any $u\in U(\hL) v$ is annihilated by $\fl_k$ for  sufficient large $k$. Now from Lemma \ref{weightupper}, for any $i=0,1,\ldots,n-1$, the weight set of $W$ module $\oplus_{j\in \bZ}M_{\lambda+\frac{i}n+j}$ is upper bounded. So is $M$, which implies that $M$ is a highest weight module.\end{proof}

\begin{lemm}\label{Z=0}Let $M$ be a  simple cuspidal $\hL$-module. Then $Z\cdot M=0$.\end{lemm}

\begin{proof} Since $M$ is simple, any $y\in Z_{\bar 0}$ acts as a scalar multiplication on $M$. We know that  $M$ has finite length as a module over $\Vir=W\oplus   \bC z$, hence $M$ has a simple cuspidal $\Vir$ module, which implies $z M=0$. 
	
	For any $z'\in Z_{\bar 1}$, since $z'M$ is a submodule of $M$, we have $z'M=0$ or $z'M=M$. From ${z'}^2M=\frac 12 [z',z']M=0$, we deduce that $z' M=0$. So
\begin{equation}\label{zero-0}z\cdot M=Z_{\bar 1}\cdot M=0.\end{equation}

 Considering $M$ as   a    cuspidal $\Vir$ module, we can find an  $m\in \bN$, such that $\Omega_{k, s}^{(m)}\in \ann_{U(\hL)}(M)$. By the same computations as for  (\ref{Omega}),  for all $k,p\in \bZ$ with $k\ne p-3,p-2,\ldots, p+m$, on $M$ we have

\begin{align}
0\equiv& [\Omega_{k,p-1}^{(m)},x(-k)]-2[\Omega_{k,p}^{(m)},x(-k-1)]+[\Omega_{k,p+1}^{(m)},x(-k-2)]\nonumber\\
&-[\Omega_{k+1,p-1}^{(m)},x(-k-1)]+2[\Omega_{k+1,p}^{(m)},x(-k-2)]-[\Omega_{k+1,p+1}^{(m)},x(-k-3)]\nonumber\\
\equiv &\sum\limits_{i=0}^{m+2}(-1)^i\binom{m+2}{i}(1-d)x(-i)\fl_{p-1+i}\\
&+\big(\sum_{i=1}^{n_{-1}} \frac{k^3-k}{12} \rho_i(x)z_{-1,i}+\sum_{i=1}^{n_0}(k^2-k) B_i(\partial,x)z_{0,i}\big)\fl_{p-1}\nonumber\\ &-\big(\sum_{i=1}^{n_{-1}} \frac{(k+1)^3-(k+1)}{12} \rho_i(x)z_{-1,i}+\sum_{i=1}^{n_0}((k+1)^2-(k+1)) B_i(\partial,x)z_{0,i})\big)\fl_{p-1}\nonumber\\
\equiv &\sum\limits_{i=0}^{m+2}(-1)^i\binom{m+2}{i}(1-d)x(-i)\fl_{p-1+i}\\
&-\big(\sum_{i=1}^{n_{-1}} \frac{k^2+k}{4} \rho_i(x)z_{-1,i}-2\sum_{i=1}^{n_0}k B_i(\partial,x)z_{0,i}\big)\fl_{p-1}\nonumber \mod \ann_{U(\hL)}(M),
\end{align}
 where $z_{i,j}:=z_{i,j,\bar 0}, \rho_i:=\rho_{i,\bar 0}, B_{i}:=B_{i,\bar 0}, n_{i}:=n_{i,\bar 0}$.
Since $\fl_{p-1}M\ne0$  for some $p$, we deduce that  \begin{equation}\label{zero-1}\sum_{i=1}^{n_{-1}} \rho_i(x)z_{-1,i}, \sum_{i=1}^{n_0} B_i(\partial,x)z_{0,i}\in \ann_{U(\hL)}(M),\forall x\in \dg.\end{equation}

Using (\ref{zero-1}) and the same computations as for (\ref{Omega}), we get

\begin{align}
0\equiv& [\Omega_{k,p-1}^{(m)},x(a+1)]-2[\Omega_{k,p}^{(m)},x(a)]+[\Omega_{k,p+1}^{(m)},x(a-1)]\nonumber\\
&-[\Omega_{k+1,p-1}^{(m)},x(a)]+2[\Omega_{k+1,p}^{(m)},x(a-1)]-[\Omega_{k+1,p+1}^{(m)},x(a-2)]\nonumber\\
\equiv &\sum\limits_{i=0}^{m+2}(-1)^i\binom{m+2}{i}(1-d)x(a+k+1-i)\fl_{p-1+i}\,\,\mod \ann_{U(\hL)}(M)\nonumber.
\end{align}

Hence \begin{align}\label{l-x}\sum\limits_{i=0}^{m+2}(-1)^i\binom{m+2}{i}x(a-i)\fl_{p+i}\in \ann_{U(\hL)}(M),\forall x(a)\in \fg,p\in \bZ.\end{align}

Now for any $a,b\in \frac1n\bZ$ and $j\in\bZ$  with $a+b+j\ne 0,1,\cdots, m+2$, using (\ref{zero-1}) we have
$$\begin{aligned}\sum\limits_{i=0}^{m+2}(-1)^i&\binom{m+2}{i}x(a-i)[\fl_{p+i-j},y(b+j)]\\
	=&[\sum\limits_{i=0}^{m+2}(-1)^i\binom{m+2}{i}x(a-i)\fl_{p+i-j},y(b+j)]\\
	&-\sum\limits_{i=0}^{m+2}(-1)^i\binom{m+2}{i}[x,y](a+b-i+j)\fl_{p+i-j}\in \ann_{U(\hL)}(M),\end{aligned}$$
	i.e., for all $j\ne -a-b,-a-b+1,\ldots,-a-b+m+2$,
we see that
\begin{align}&\sum\limits_{i=0}^{m+2}(-1)^i\binom{m+2}{i}x(a-i)[\fl_{p+i-j},y(b+j)]\nonumber \\&=\sum\limits_{i=0}^{m+2}(-1)^i\binom{m+2}{i}x(a-i)(b+j+(p+i-j)d)y(b+p+i)\in  \ann_{U(\hL)}(M)\label{4.7}.\end{align}

Since the right-hand side of (\ref{4.7}) is a polynomial of $j$, we can remove the condition for $j$ in (\ref{4.7}) to yield 

\begin{equation}\label{4.8}\sum\limits_{i=0}^{m+2}(-1)^i\binom{m+2}{i}x(a-i)[\fl_{p+i},y(b)]\in  \ann_{U(\hL)}(M), \forall a, b\in  \frac1n\bZ,  p\in\bZ.\end{equation}
 
Now from (\ref{l-x}) and (\ref{4.8}), for any $x(a),y(-a)\in \dg$ with $a\ne p,p+1,\ldots,p+m+2$,  we have  \begin{align*}0&\equiv \sum_{i=1}^{n_{-1}}\frac{1-4a^2}{24}\rho_i([x,y])z_{-1,i}+\sum_{i=1}^{n_0} a B_i(x,y)z_{0,i}+\sum_{i=1}^{n_1}\dot\alpha_i(x,y)z_{1,i}\\ &\equiv [\sum\limits_{i=0}^{m+2}(-1)^i\binom{m+2}{i}x(a-i)\fl_{p+i},y(-a)]-\sum\limits_{i=0}^{m+2}(-1)^i\binom{m+2}{i}x(a-i)[\fl_{p+i},y(-a)]
\\ &-\sum\limits_{i=0}^{m+2}(-1)^i\binom{m+2}{i}[x,y](-i)\fl_{p+i}\mod \ann_{U(\hL)}(M).\end{align*}

Hence
\begin{equation}\label{zero-2} \sum_{i=1}^{n_0}  B_i(x,y)z_{0,i}, \sum_{i=1}^{n_1}\dot\alpha_i(x,y)z_{1,i}\in \ann_{U(\hL)}(M),\forall x,y\in \dg.\end{equation}

From (\ref{zero-0}), (\ref{zero-1}) and (\ref{zero-2}), we have $Z\cdot M=0$.
\end{proof}

Now Theorem \ref{thm2} follows from Theorem \ref{noncuspidal}, Lemma \ref{Z=0}, Theorem \ref{cuspidal} and Theorem \ref{AL}.

In order to apply Theorem \ref{thm2}, one has to first find   all  simple finite dimensional modules $V$ over the finite dimensional Lie superalgebra $\ddg$. Then construct the loop module $\Gamma(V,\lambda)$ for any $\lambda\in\bC$ following the steps in (3.6)-(3.8), and find its simple subquotient modules. 

\begin{section}{Examples}\end{section}

 In this section, we shall apply our  Theorems \ref{thm1} and \ref{thm2} to some Lie algebras to recover many known results. We shall also have  new Lie (super)algebras and new results in  each of these examples.
 
 \
 
\begin{exam}\label{ex1} Let $\dg=\bC e$ be the  $1$-dimensional trivial  Lie algebra, $d(e)=\beta e, \sigma(e)=\omega_n^ie $ for some $\beta\in \bC$, $n\in \bN$ and $i\in \{0,1,\ldots,n-1\}$ with $\beta \ne 1$ and $\gcd(n,i)=1$. Then the Lie algebra $\fL=W\ltimes et^{\frac in}A$ has brackets
 \begin{equation}[\fl_i, \fl_j]=(j-i)\fl_{i+j}; [\fl_j,et^{\frac in+k}]=(\frac in+k+j\beta)et^{\frac in+k+j},[et^{\frac in}A,et^{\frac in}A]=0.\end{equation}
 By Theorem \ref{thm1}, we compute  
 $$\begin{aligned}\dim H^2(\fL,\bC)_{(-1)}=&\left\{\begin{array}{lr}1,& \text{\rm if }\beta=-1,i=0,\\0,& {\rm otherwise}, \end{array}\right.\\
 	\dim H^2(\fL,\bC)_{(0)}=&\left\{\begin{array}{lr}3,& \text{\rm if }\beta=i=0,\\ 2, &\text{\rm if }\beta=0,\frac in=\frac 12,\\ 1, &{\rm otherwise}, \end{array}\right.\end{aligned}$$ and $\dim H^2(\fL,\bC)_{(1)}=\dim H^2(\dg,\bC)^{d,\sigma}=0.$ 
 	Hence $$\dim H^2(\fL,\bC)=\left\{\begin{array}{lr}2,&\text{\rm if } \beta=-1,i=0, \\3,&\text{\rm if }\beta=i=0,\\ 2,&\text{\rm if }\beta=0,\frac in=\frac 12,\\1,& {\rm otherwise}. \end{array}\right.$$

A finite dimensional simple module  over $\ddg$ is 1-dimensional. From Theorem \ref{thm2}, we know that any simple quasi-finite $\hL$-module is a highest (or lowest) weight module or a module of intermediate series (i.e. a weight module with all 1-dimensional weight spaces). Let us determine the $\hL$-modules of intermediate series.

{\bf Case 1: $\beta\ne0$.} Simple finite dimensional $\ddg$-modules are $V(\mu)=\bC v$ for some $\mu\in\bC$ with $\partial v=\mu v, \dg v=0.$ Simple cuspidal $\hL$-modules $X$ are simply simple cuspidal Vir-modules (with $\mathfrak{g} X=0$ where $\mathfrak{g}$ is defined in Sect.1).

{\bf Case 2: $\beta=0$.} 
A finite dimensional simple module  over $\ddg$ is of the form  $V(\mu_1, \mu_2)=\bC v$ for some $\mu_i\in\bC$ with $\partial v=\mu_1 v,  e v= \mu_2 v\ne0.$ Simple cuspidal $\hL$-modules are simple subquotients of $V(\mu_1, \mu_2; \lambda)=\bC[t^{\pm \frac1n}]$ for some $\lambda,\mu_i\in\bC$  with  
$$\aligned &Z\cdot V(\mu_1, \mu_2; \lambda) =0, \\ &e(\frac {i}n+k_1)\cdot t^{\frac{k_2}{n}}=\mu_2t^{\frac {i}n+k_1+\frac{k_2}{n}},\\
&\fl_{k_1}\cdot t^{\frac{k_2}{n}}=(\lambda+\frac{k_2}{n}+k_1\mu_1)t^{ {k_1}+\frac{k_2}{n}}, \forall k_1,k_2\in\bZ.
\endaligned$$
The loop $\hL$-module $V(\mu_1, \mu_2; \lambda)$ is simple since $\mu_2\ne0$.

The universal central extensions and simple quasi-finite representations for $\fL$ with $\frac in= 1$ were given in \cite{GJP,L}. The Lie algebras
$\hL$ is  the twisted Heisenberg-Virasoro algebra if $\beta=i=0$,  the mirror Heisenberg-Virasoro algebra if $\beta=0,\frac in=\frac 12$, and $W(2,2)$ if $\beta=-1,i=0$. All other Lie algebras $\hL$ in this example were not known in the literature. In particular, all Lie algebras $\hL$ for $\frac in\ne 1$ or $\frac12$ are new.
\end{exam}

\begin{exam} Let $p>1$, let $\dg=\bC x_1+\cdots+\bC x_{p-1}$ be the commutative Lie algebra of dimension $p-1$, $d=0$, $\sigma(x_i)=\omega_p^i x_i$ for $ i=1,2,\ldots,p-1$. Then $\sigma$ has order $p$ and $\dg=\oplus_{i=1}^{p-1}{\dg}_{[i]}$ with $\dg_{[i]}=\bC x_i$. The Lie algebra $\hL(\dg,d,\sigma)$ is called gap-p Virasoro algebra (in a slightly different form) which was studied in \cite{X}. Note that $\ddg$ is commutative in this case. By Theorem \ref{thm1}, we know that $$\dim H^2(\fL,\bC)=\dim (\Inv(\ddg))^{\sigma}=k+1\text{ for }p=2k {\text{ or }}2k+1.$$ 
	Actually the matrix of skew-symmetric bilinear form in $(\Inv(\dg))^{\sigma}$ is of the form
	$$\sum_{i=1}^{\lfloor{\frac{p-1}2}\rfloor} a_i(E_{i, n-i+1}+E_{ n-i+1,i}), \text{ where }a_i\in\bC.$$
	
	Since any finite dimensional simple module  over $\ddg$ is 1-dimensional, from Theorem \ref{thm2}, we know that any simple quasi-finite $\hL$-module is a highest (or lowest) weight module or a module of intermediate series (i.e. a weight module with all 1-dimensional weight spaces).
	
	Let us determine the $\hL$-modules of intermediate series. Simple finite dimensional  $\ddg$-modules are of the form $V(\mu_1, \mu_2,\cdots, \mu_{p})=\bC v$ for some $\mu_i\in\bC$ with $\partial v=\mu_p v, e_i v=\mu_i v.$ Simple cuspidal $\hL$-modules are simple subquotients of $V(\mu_1, \mu_2,\cdots, \mu_{p}; \lambda)=\bC[t^{\pm \frac1p}]$ for some $\lambda,\mu_i\in\bC$  with  
	$$\aligned & Z\cdot V(\mu_1, \mu_2,\cdots, \mu_{p}; \lambda) =0, \\ &x_j(\frac {j}p+k_1)\cdot t^{\frac{k_2}{p}}=\mu_jt^{\frac {j}n+k_1+\frac{k_2}{p}},\\
	&\fl_{k_1}\cdot t^{\frac{k_2}{p}}=(\lambda+\frac{k_2}{p}+k_1\mu_p )t^{ {k_1}+\frac{k_2}{p}}, \forall k_1,k_2\in\bZ.
	\endaligned$$
	It is not hard to determine the necessary and sufficient conditions for the loop $\hL$-module $V(\mu_1, \mu_2,\cdots, \mu_{p}; \lambda)$ to be simple.
	
	One may easily notice that the above results on $\hL$-modules of  intermediate series are quite different from that in \cite{X}.
\end{exam}

\begin{exam}Let $\dg=\bC^{0|1}$ be the Lie superalgebra of dimension 1,  $d(1)=\beta\ne 1, \sigma(1)=\omega_n^i $ for some $\beta\in \bC$, $n\in \bN$ and $i\in \{0,1,\ldots,n-1\}$ with $\gcd(n,i)=1$. the Lie algebra $\fL$ has the same basis and brackets as in Example 1 but different parities, in particular, $ t^{\frac in}A$ has odd parity and $[ t^{\frac in}A, t^{\frac in}A]=0$. By Theorem \ref{thm1},  we compute 
	$$\begin{aligned}\dim H^2(\fL,\bC^{1|1})_{\bar 0,(-1)}=&\left\{\begin{array}{lr}1,&\text{\rm if } \beta=-1,i=0,\\0,& {\rm otherwise}, \end{array}\right.\\
	\dim H^2(\fL,\bC^{1|1})_{\bar 0,(0)}=&\left\{\begin{array}{lr}2,& \text{\rm if }\beta=i=0,\\ 1, &{\rm otherwise}, \end{array}\right. \\
	\dim H^2(\fL,\bC^{1|1})_{\bar 0,(1)}=&\dim H^2(\dg,\bC^{1|1})_{\bar 0}^{d,\sigma}=\left\{\begin{array}{lr}1,&\text{\rm if }\beta=\frac 12,\frac in\in \{0,\frac 12\},\\ 0,& {\rm otherwise.}\end{array}\right.\end{aligned}$$ 
	 Hence $$\dim H^2(\fL,\bC^{1|1})_{\bar 0}=\left\{\begin{array}{lr}2,& i=0, \beta\in\{0,-1\}, \\ 2,&\beta=\frac 12,\frac in\in \{0,\frac 12\},\\1,& {\rm otherwise}. \end{array}\right.$$

Since any finite dimensional simple module  over $\ddg$ is trivial 1-dimensional, from Theorem \ref{thm2}, we know that any simple quasi-finite $\hL$-module is a highest (or lowest) weight module or a module of intermediate series with trivial action of the odd part (i.e. a weight module with all 1-dimensional weight spaces over $W$).

If $\beta=\frac 12$ and $ \frac in\in \{0,\frac 12\}$, the Lie algebra $\hL=\hL(\dg,d,\sigma)$ is called Fermion-Virasoro algebras defined and studied in \cite{DCL,XZ}. Besides the Fermion-Virasoro algebras, all other Lie superalgebras  in this example were not seen in the literature.
 
\end{exam}

\begin{exam} Let $\dg=\bC h+\bC e$ be a $2$-dimensional Lie superalgebra with $h\in \dg_{\bar 0}, e\in \dg_{\bar 1}$, and $h=[e,e], [h,e]=0$. For any given $d$ and $\sigma$, there exists 
$i,n\in \bZ_+, \beta\in \bC$, such that $d(h)=2\beta h,d(e)=\beta e, \sigma(e)=\omega_n^{i} e, \sigma(h)=\omega_n^{2i} h$ with $\beta\ne 1$ or $\frac 12$,  $\gcd(n,i)=1, i\in \{0,1,\ldots,n-1\}$. Then $\fL=W\ltimes \big((e\otimes t^{\frac in}A)\oplus (h\otimes t^{\frac {2i}n}A)\big)$. From Theorem \ref{thm1}, we compute that $\dim H^2(\fL,\bC^{1|1})_{\bar 0,(1)}=0$, and 
	$$\begin{aligned}
\dim H^2(\fL,\bC^{1|1})_{\bar 0,(-1)}=\left\{\begin{array}{lr}1,& \text{\rm if }i=0,\beta=-1,\\1,& \text{\rm if }\beta=-\frac 12,\frac in\in\{0,\frac 12\},\\ 0, &{\rm otherwise}, \end{array}\right. \\
\dim H^2(\fL,\bC^{1|1})_{\bar 0,(0)}=\left\{\begin{array}{lr}3,& \text{\rm if }i=\beta=0,\\ 2,& \text{\rm if }\frac in\in\{\frac 13,\frac 23\},\beta=0,\\ 1,  &{\rm otherwise}. \end{array}\right.\end{aligned}$$ 
Hence 
\begin{equation*} \dim H^2(\fL,\bC^{1|1})_{\bar 0}=\left\{\begin{array}{lr}2,&\text{\rm if } i=0,\beta=-1,\\2,&\text{\rm if }\beta=-\frac 12,\frac in\in\{0,\frac 12\},\\3,& \text{\rm if }i=\beta=0,\\ 2,& \text{\rm if }\frac in\in\{\frac 13,\frac 23\},\beta=0,\\ 1,  &{\rm otherwise}.\end{array}\right.\end{equation*}

Note that any finite dimensional simple module over $\ddg$ is 1-dimensional with $\dg$ acting as zero if  $\beta\ne 0$, and $1$-dimensional or $2$-dimensional if $\beta=0$.
  From Theorem 2, any simple quasi-finite $\hL$-module is a highest(lowest) weight $\hL$-module or a cuspidal module with weight multiplicity $\le 2$. 
  Let us determine the cuspidal $\hL$-modules $\tilde V$. 
  
  	{\bf Case 1: $\mathfrak{g} \tilde V=0$.} The simple $\hL$-modules $\tilde V$  are simply simple cuspidal Vir-modules.
  
  {\bf Case 2: $\mathfrak{g} \tilde V\ne0$.} We deduce that $\beta=0$ and $ \tilde V= \tilde V(\mu_1,\mu_2; \lambda)= \tilde V_{\bar 0}\oplus  \tilde V_{\bar 1}$ where $\tilde V_{\bar 0}=  v_0\bC[t^{\pm \frac1n}],    \tilde V_{\bar 1}=  v_1\bC[t^{\pm \frac1n}]$  for some $\lambda,\mu_i\in\bC$ with $\mu_2\ne0$, subject to the actions:
   $$\aligned &Z\cdot \tilde V(\mu_1,\mu_2; \lambda) =0, \\
   & e(\frac {i}n+k_1)\cdot v_{\bar 0}t^{\frac{k_2}n}= v_{\bar 1}t^{\frac {i}n+k_1+\frac{k_2}n}, \\
   &e(\frac {i}n+k_1)\cdot v_{\bar 1}t^{\frac{k_2}n}= \frac{\mu_2}2v_{\bar 0}t^{\frac {i}n+k_1+\frac{k_2}n},\\ &h(\frac {2i}n+k_1)\cdot v_{\bar k}t^{\frac{k_2}{n}}=\mu_2v_{\bar k}t^{\frac {2i}n+k_1+\frac{k_2}{n}},\\
  &\fl_{k_1}\cdot v_{\bar k}t^{\frac{k_2}{n}}=(\lambda+\frac{k_2}{n}+k_1\mu_1)v_{\bar k}t^{ {k_1}+\frac{k_2}{n}}, \forall k, k_1, k_2\in\bZ.
  \endaligned$$
  The loop $\hL$-module $\tilde V(\mu_1, \mu_2; \lambda)$ is simple since $\mu_2\ne0$.

  We remark that for $\beta=-\frac 12,\frac in=\frac 12$, $\hL$ is exactly the $N=1$ BMS superalgebra defined in \cite{BBM}, and representation theory for BMS superalgebra has been extensively studied, see \cite{LPXZ2} and references therein. All other Lie superalgebras in this example were not seen in the literature.
\end{exam}

\begin{exam} Let $\dg$ is a finite dimensional simple Lie superalgebra, $d=0$, and $\sigma$ be an order $n$ automorphism of $\dg$. For any $B\in \Inv(\ddg)$, $B(\partial, \dg)=B(\partial,[\dg,\dg])=B([\partial,\dg],\dg)=0$, Hence $\dim \Inv(\ddg)=1+\dim \Inv(\dg)$. Since $\dim \Inv(\dg)\le 1$,  from Theorem \ref{thm1}, we have $1\le \dim H^2(\fL,\bC^{1|1})_{\bar 0}\le 2$. Actually $H^2(\fL,\bC^{1|1})_{\bar 0}= 2$ if $\dg$ is a finite dimensional simple Lie algebra, giving twisted affine-Virasoro algebras.

	From Theorem \ref{thm2}, any quasi-finite simple $\hL$-module that is not a highest (or lowest) weight module is a simple subquoitent of a loop module. Since twisted Affine Kac-Moody superalgebra can be realized as a fixed point subalgebra of a nontwisted Affine Kac-Moody superalgebra (it means $n=1$), it generalizes the results for nontwisted Affine-Virasoro algebras  in \cite{LPX} to twisted Affine-Virasoro superalgebras (it means $n>1$). 
	
	Let us explain the loop modules $\Gamma(V, \lambda)$. Let $V$ be a finite-dimensional simple module over $\ddg$ which is simply a simple $\dg$-module with $\partial V=0$. If $\dg V=0$, then $\Gamma(V, \lambda)$ is simply a Vir-module of intermediate series which is clear. Now we consider the case that $\dg V\ne 0$.
	Then $\Gamma(V, \lambda)=V\otimes \bC[t^{\pm \frac1n}]$ with 
		$$\aligned  &x(\frac {i}n+k_1)\cdot vt^{\frac{k_2}{n}}=(xv)t^{\frac {i}n+k_1+\frac{k_2}{n}},\\
	&\fl_{k_1}\cdot vt^{\frac{k_2}{n}}=(\lambda+\frac{k_2}{n})vt^{ {k_1}+\frac{k_2}{n}},
	\endaligned$$
	for all $k_1, k_2\in\bZ, x\in\dg_{[i]}$. It is easy to see that $\Gamma(V, \lambda)$ is a simple $\hL$-module, and it is not straightforward to determine the simple subquotients of $\Gamma(V, \lambda)$ for $n>1$.
	\end{exam}

\begin{exam}\label{5.6} Let $\dg=\so(m)\ltimes \bC^m$ , $d|_{\so(m)}=0, d|_{\bC^m}=-{\rm id},  \sigma=1$,  where $\so(m)$ is the orthogonal Lie algebra consists of all $m\times m$ skew-symmetric matrices over $\bC$, and $\bC^m$ is the $\so(m)$ module with actions of matrix multiplication. Then $\fL=W\ltimes (\so(m) \ltimes \bC^m)\otimes A$ with brackets
\begin{align*}&[\fl_l,\fl_j]=(j-l)\fl_{l+j}, &[\fl_l, J(j)]=jJ(l+j),\\
 &[\fl_l, \eta(j)]=(j-l)\eta(j+l),& [x(l),y(j)]=[x,y](l+j),\end{align*}
 for all $x,y\in \so(m)\ltimes \bC^m ; j,l\in \bZ, J\in \so(m),\eta\in \bC^m$. The Lie algebras $\fL$ were defined and studied in \cite{BG} as an infinite dimensional extension of the conformal Galilei algebra.

Note that $\so(1)=0,\dim \so(2)=1$ and $\so(3)\cong \sl_2$, $\so(4)\cong \so(3)\oplus \so(3)$, $\so(m)$ are simple for $m\ge 5$, and $[\so(m),\bC^m]=\bC^m$ for   $m\ge 2$.

For $m=1$, the Lie algebra $\fL=W(2,2)$ is studied in Example \ref{ex1}.

For $m\ge 2$, from Theorem \ref{thm1}, we know that $ H^2(\fL,\bC)_{(-1)}= H^2(\fL,\bC)_{(1)}=0$, and

$$\dim H^2(\fL,\bC)=\dim H^2(\fL,\bC)_{(0)}=\left\{\begin{array}{lr}3,& \text{\rm if }m=2, \\ 3,& \text{\rm if }m=4,\\2,& {\rm otherwise}. \end{array}\right.$$

From Theorem \ref{thm2}, for $m\ge 2$, it is easy to see that for any simple quasi-finite module $M$, $(\bC^m\otimes A) \cdot M=0$, and $M$ is a quasi-finite simple module over $\hL(\so(m),0,1)$ which is the case for $n=1$ in Example 5. 

 For $m=2$, $\hL$ is called the  planar Galilean conformal algebra in \cite{Ai}, and it's structure and representation were extensively studied, see \cite{GLP, GG}and references therein. All other results in this example were not seen in the literature.\end{exam}

\begin{exam} In the previous examples, $d$ is always diagonalizable. Now we give an example in which $d$ is not diagonalizable. Let $\dg=\bC e_1+\bC e_2$ be a $2$-dimensional abelian Lie  algebra. Take  $\sigma$ as the scalar map by $\omega_n^{i}$ where  $i\in \{0,1,\ldots,n-1\}$ with $\gcd(n,i)=1$. Take the derivation $d$ as  $d(e_1)=\beta e_1,d(e_2)=\beta e_2+e_1$ where $\beta\in\mathbb{C}$ with $\beta\ne1$. Then $\fL=W\ltimes \big((e_1\otimes t^{\frac in}A)\oplus (e_2\otimes t^{\frac {i}n}A)\big)$ and 
	\begin{align*}
		&[\fl_i,e_1\otimes t^a]=(a+i\beta)e_1\otimes t^{i+a},\\
		&[\fl_i,e_2\otimes t^a]=\big((a+i\beta)e_2+ie_1\big)\otimes t^{i+a},
	\end{align*}
	 From Theorem \ref{thm1}, we compute that  
	$$\begin{aligned}
		\dim H^2(\fL,\bC^{1|1})_{\bar 0,(-1)}=\left\{\begin{array}{lr}1,& \text{\rm if }\beta=-1, n=1,\\ 0, &{\rm otherwise}, \end{array}\right. \\
		\dim H^2(\fL,\bC^{1|1})_{\bar 0,(0)}=\left\{\begin{array}{lr}3,& \text{\rm if }\beta=0, n=1,\\
			2,& \text{\rm if }\beta=0,\frac in=\frac 12,\\ 1, &{\rm otherwise}, \end{array}\right.\\
		\dim H^2(\fL,\bC^{1|1})_{\bar 0,(1)}=\left\{\begin{array}{lr} 
			1,& \text{\rm if }\beta= \frac 12,\frac in\in\{1,\frac 12\}\\ 0, &{\rm otherwise}, \end{array}\right.\end{aligned}$$ 
	Hence 
	\begin{equation*} \dim H^2(\fL,\bC^{1|1})_{\bar 0}=\left\{\begin{array}{lr} 3,& \text{\rm if }\beta=0, n=1,
						\\ 2,& \text{\rm if } n=1,\beta\in\{1, \frac12\},
				 				\\ 2,& \text{\rm if }\frac in=\frac 12,\beta\in\{0, \frac 12 \},
			\\ 1,  &{\rm otherwise}.\end{array}\right.\end{equation*}
	
	Note that any finite dimensional simple module over $\ddg$ is 1-dimensional.
	From Theorem 2, any simple quasi-finite $\hL$-module is a highest(lowest) weight $\hL$-module or a cuspidal module with weight multiplicity $1$. Let us determine the $\hL$-modules of intermediate series. 
	
	{\bf Case 1: $\beta\ne0$.} Simple finite dimensional $\ddg$-modules are $V(\mu)=\bC v$ for some $\mu\in\bC$ with $\partial v=\mu v, \dg v=0.$ Simple cuspidal $\hL$-modules $X$ are simply simple cuspidal Vir-modules with $\mathfrak{g} X=0$.
	
		{\bf Case 2: $\beta=0$.} Simple finite dimensional $\ddg$-modules are $V(\mu_1, \mu_2)=\bC v$ for some $\mu_i\in\bC$ with $\partial v=\mu_1 v, e_1 v=0, e_2 v=\mu_2 v.$ Simple cuspidal $\hL$-modules are simple subquotients of $V(\mu_1, \mu_2, \lambda)=\bC[t^{\pm \frac1n}]$ for some $\lambda,\mu_i\in\bC$  with  
		$$\aligned &e_1(\frac {i}n+k_1)\cdot t^{\frac{k_2}n}=0,\,\,\, Z\cdot V(\mu_1, \mu_2, \lambda) =0, \\ &e_2(\frac {i}n+k_1)\cdot t^{\frac{k_2}{n}}=\mu_2t^{\frac {i}n+k_1+\frac{k_2}{n}},\\
		&\fl_{k_1}\cdot t^{\frac{k_2}{n}}=(\lambda+\frac{k_2}{n}+k_1\mu_1)t^{ {k_1}+\frac{k_2}{n}}, \forall k_1, k_2\in\bZ.
		\endaligned$$
	The loop $\hL$-module $V(\mu_1, \mu_2, \lambda)$ is simple if and only if $\mu_2\ne0$.

	 We remark that the Lie algebras  $\hL$ in this example were not known in the literature.
\end{exam}

Theorem \ref{thm1} provides an applicable method to compute the universal central extension of the Lie algebra $\fL(\dg, d, \sigma)$. From the computations in the above examples, in general, it is straightforward to compute the space 
$\Big(\dg/\big((d+1)\dg+[(d+\frac 12)\dg,\dg]+ [\dg,[\dg,\dg]]\big)\Big)^{\sigma}$, $H^2(\dg,\bC^{1|1})_{\bar 0}^{d,\sigma}$ and $ {(\Inv(\ddot{\fg}))}^\sigma$, but involving a lot of computations for many $\dg, d$ and $\sigma$.

\

\noindent {\bf Acknowledgement}. R. L\"u is partially supported by National Natural Science Foundation of China (Grant No.
12271383), K. Zhao is  partially supported
by   NSERC
(311907-2020).




\

Rencai L\"u,  Department of Mathematics, Soochow University, Suzhou 215006, P. R. China. Email: rlu@suda.edu.cn

\

Xizhou You, Department of Mathematics, Soochow University, Suzhou 215006, P. R. China.  xzyou@alu.suda.edu.cn

\

Kaiming Zhao,  School of Mathematics and Statistics,
Xinyang Normal University,
Xinyang 464000, P. R. China, and Department of Mathematics, Wilfrid Laurier University, Waterloo, ON, Canada N2L3C5. Email address: kzhao@wlu.ca


\end{document}